\documentclass[11pt]{article}
\pdfoutput=1
\usepackage{amsmath}
\usepackage{amsfonts}
\usepackage{amssymb}
\usepackage{amsthm}
\usepackage[pdftex]{graphicx}
\usepackage[latin1]{inputenc}
\usepackage{empheq}
\usepackage{epstopdf}
\usepackage{placeins}
\usepackage{gensymb}
\usepackage{float}
\usepackage{verbatim}
\usepackage{units}
\usepackage{algorithm2e}
\usepackage{hyperref}
\usepackage{subfigure}
\usepackage{pgfgantt}

\newcommand{\vertiii}[1]{{\left\vert\kern-0.25ex\left\vert\kern-0.25ex\left\vert #1 
    \right\vert\kern-0.25ex\right\vert\kern-0.25ex\right\vert}}

\newtheorem{theorem}{Theorem}
\newtheorem{lemma}{Lemma}
\newtheorem{proposition}{Proposition}

\title{Numerical upscaling of discrete network models}
\author{G. Kettil, A. M\r{a}lqvist, A. Mark, M. Fredlund, K. Wester, F. Edelvik}
\date{\today}

\begin{document}

\maketitle
\thispagestyle{empty}

\begin{abstract}
\noindent In this paper a numerical multiscale method for discrete networks is presented. The method gives an accurate coarse scale representation of the full network by solving sub-network problems. The method is used to solve problems with highly varying connectivity or random network structure, showing  optimal order convergence rates with respect to the mesh size of the coarse representation. Moreover, a network model for paper-based materials is presented. The numerical multiscale method is applied to solve problems governed by the presented network model.
\end{abstract}

\section{Introduction}
Network structures are used to model a wide variety of phenomena, such as flow in porous media, traffic flows, elasticity of materials, body deformation in computer graphics, molecular dynamics, and fiber materials. In these applications, the microscale behaviour determines the macroscale properties of the system. Often a full microscale model is difficult or impossible to work with because of the vast computational complexity. Therefore, there is an interest in constructing coarser, but still accurate, representations of the entire system. Such a procedure is sometimes referred to as upscaling or homogenization. In this work a numerical upscaling method for discrete networks is presented.

There exist several numerical upscaling methods for partial differential equations (PDE) based on the idea of homogenization, such as the Heterogeneous Multiscale Method (HMM) \cite{HMM}, the Multiscale Finite Element Method (MsFEM) \cite{MsFEM}, and the more recent works \cite{Efendiev, Owhadi}. The upscaling approach presented in this paper is based on the Localized Orthogonal Decomposition Method (LOD) \cite{LOD1, LOD2}, which in turn is inspired by the Variational Multiscale Method (VMM) \cite{Hughes}. Multiscale methods applied to network problems are for instance investigated by Ewing, Ilev et al.~\cite{Ewing, Iliev} who study the heat conductivity of network materials and develop an upscaling method by solving the heat equation locally over small sub-domains. These local solutions are used to compute an effective global thermal conductivity tensor. Della Rossa et al.~investigate network models  of traffic flows \cite{DellaRossa} and derive a governing PDE for the macroscale by formulating traffic flow equations for single network nodes and interpreting the relations as finite difference approximations. The macroscale parameters are resolved using a two-scale averaging technique. Chu et al.~develop a multiscale method for networks representing flows in a porous medium \cite{Chu}. The medium is modelled as a network where nodes represent pores and edges represent throats. The conductance of each throat is assumed to be given by Hagen-Poiseuille equation, and using mass conservation equations for the flow through the network, a model for the microscale is attained. 

The numerical upscaling method proposed in this work is developed for general unstructured networks. The network is supposed to represent the microscale, and the macroscale is represented by a finite element mesh which is coarse in comparison to the fine scale network. The coarse grid does not have to be related to the network in any way except that both cover the same computational domain, and therefore the method can be applied to arbitrary network geometries. The coarse FEM grid is used to define a macroscale solution space spanned by basis functions defined at each coarse grid node as in standard FEM. The upscaling idea is to modify the coarse basis functions to account for the microscale features of the network. This is accomplished by solving local sub-network problems at each coarse basis function. The modified basis functions are thereafter used to solve a global low-dimensional system resulting in an accurate macroscale solution. The method leads to modified basis functions that decay exponentially, and hence localization of the local sub-network problems can be utilized, reducing the computational  cost considerably while preserving optimal convergence rates.

Moreover, this paper includes a two-dimensional  network model, which can be used to model paper-based materials in form of fiber networks. The macroscale mechanical properties of paper-based materials are of great interest. Paper is a heterogeneous material built up of fibers bonded together into a network structure. The mechanical properties of paper depend primarily on the properties of the fibers and the bonds between them. In \cite{Mark1, Mark2, Lic}, computational fluid dynamics and advanced contact modeling are used to simulate the paper forming process. One future aim is to utilize that framework together with the proposed multiscale method to create virtual fiber networks and investigate  the macroscale mechanical properties. A network representation including fibers and bonds is a suitable methodology to study the mechanical properties of paper \cite{Raisanen, Hagglund, Kulachenko}.   Moreover, the varying properties of single fibers and bonds, as well as an interest for fracture propagation simulations, call for an upscaling approach.  The presented network model  is based on forces arising at the nodes when the network is displaced, acting to restore the initial configuration. The network model is similar to lattices models like \cite{Ostoja, Beex}  where edges are represented by springs. Moreover,  angle springs between pair of edges are included. A novelty of the  network model in this work is a third type of force phenomenon resulting in an effect similar to the Poisson effect. Force equilibrium equations at each node result in a matrix equation which can be very large. For a regular network, the model converges to the linear elasticity equation when the length of the network edges tends to zero. The numerical upscaling method is applied to the network model and numerical examples are solved to demonstrate the convergence rates of the method. The examples show how the proposed numerical upscaling method resolves fine scale features which the standard FEM cannot.

The outline of this text is as follows. In Sect.~\ref{ProblemFormulation}, the general problem formulation is stated. Thereafter, in Sect.~\ref{UpscalingMethod}, the theory of the numerical upscaling method is presented. Sect.~\ref{ErrorAnalysis} contains error analysis, and in Sect.~\ref{NetworkModel}, the two-dimensional network model is described. In Sect.~\ref{NumericalExamples},  the network model together with the numerical upscaling method are applied in numerical  examples, showing the convergence rates of the proposed method. Lastly, in Sect.~\ref{Conclusions}, the conclusion and future work are discussed.

\section{Problem formulation}
\label{ProblemFormulation}
Consider a problem modelled by a network with properties governed by a  connectivity matrix $K\in\mathbb{R}^{n\times n}$. The matrix $K$ can for instance be the discrete Poisson operator describing heat conduction, the finite difference discretization of the linear elasticity operator, or represent a more complex model, such as of the mechanics of a fiber network. Let $F\in\mathbb{R}^{n}$ denote the load vector and let the solution vector be denoted $u$, belonging to a vector space $V\subset\mathbb{R}^n$. The network  problem can be stated in two equivalent ways, either:
\begin{align}
\label{FormulationI}
\begin{split}
\textnormal{Find }u:\quad\bar K u = \bar F,
\end{split}
\end{align}
or:
\begin{align}
\label{FormulationII}
\begin{split}
\textnormal{Find }u\in V:\quad v^T Ku=v^T F, \quad \forall v\in V.
\end{split}
\end{align}
In the first formulation, \eqref{FormulationI}, $\bar K$ and $ \bar F$ denotes modifications of $K$ and $F$ by explicitly including the restriction of $u$ to the space $V$, for instance by holding some nodes fixed. To ensure existence and uniqueness of the second formulation, \eqref{FormulationII}, it is assumed that $K$ is symmetric and positive definite on $V$. A matrix $K\in\mathbb{R}^{n\times n}$ is positive definite on a subset $V\subset\mathbb{R}^n$ if $v^TKv>0$ for all nonzero $v\in V$. Moreover, a symmetric positive definite matrix $K$ constitutes a scalar product $\langle u,v\rangle=u^TKv$ on $V$, a property that will be used later.

In Fig. \ref{ThreeNet} three examples of networks are shown. The network in Fig. \ref{nwI} exemplifies a finite difference grid for the unit square, with $K$ as the resulting discretization of the linear elasticity operator. This problem setup can be used to find the node displacements $u$ under applied node forces $F$. To attain a solvable system $Ku=F$, some degrees of freedom have to be prescribed, resulting in the restricted solution space $V$. The network in Fig. \ref{nwII} can represent a conductive medium, governed by the discrete Poisson equation. The temperature at each node is contained in $u$. The network in Fig. \ref{nwIII} is a fiber network building up a paper sheet. The fibers are modelled as chains of edges connected at nodes with bonds between fibers at common network nodes.

\begin{figure*}[ht!]
\centering
\subfigure[]{\includegraphics[trim={3.1cm 1.2cm 2.6cm 0.8cm},clip,width=0.3\linewidth]{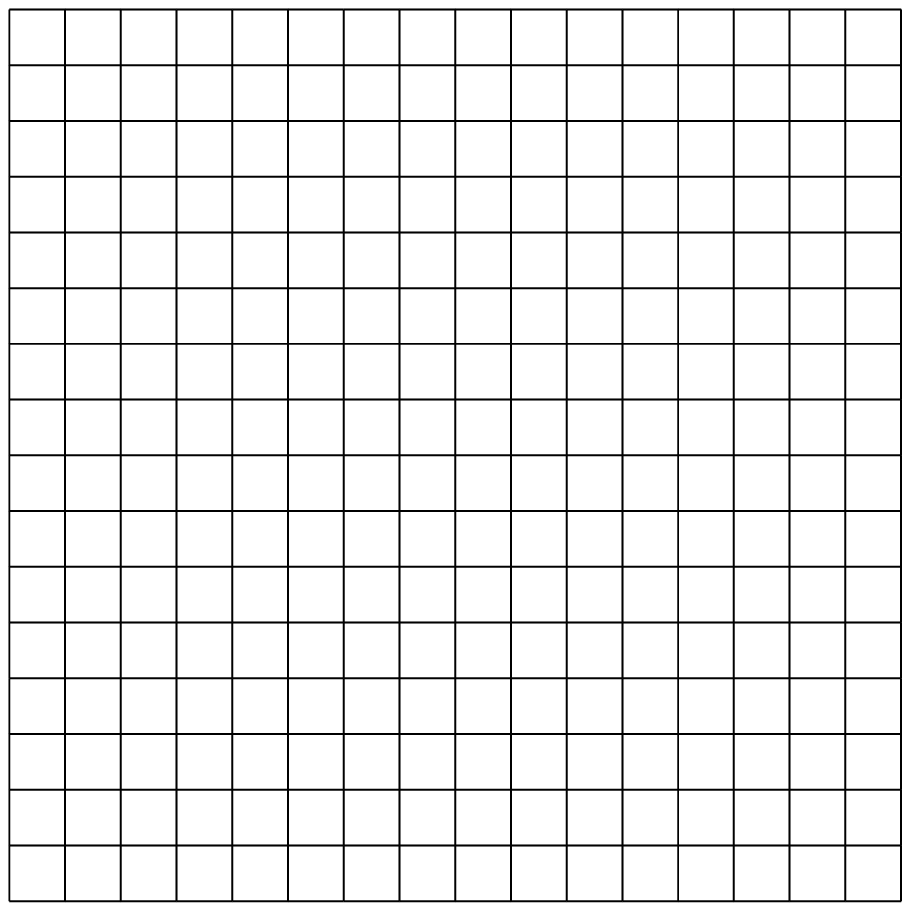}\label{nwI}}
\hspace{0.1cm}
\subfigure[]{\includegraphics[trim={3.7cm 1.2cm 3.1cm 0.8cm},clip,width=0.3\linewidth]{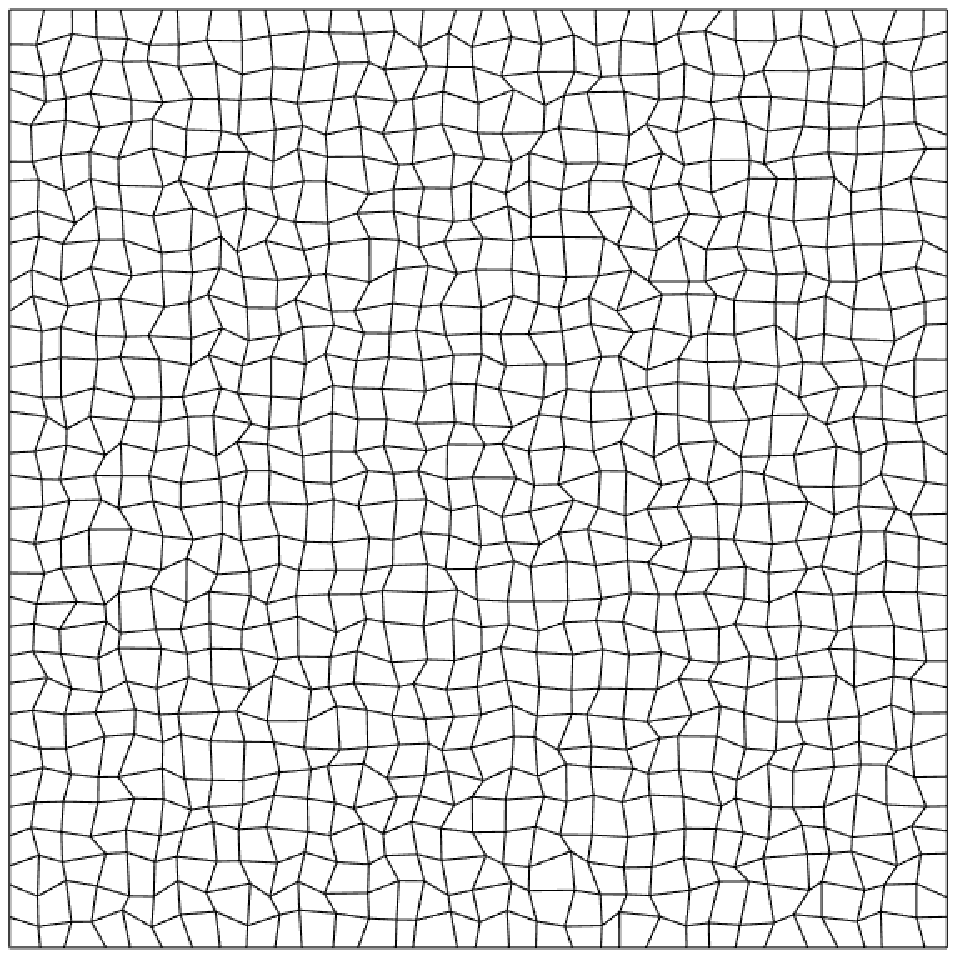}\label{nwII}}
\hspace{0.1cm}
\subfigure[]{\includegraphics[trim={1.0cm 1.1cm 1.5cm 0.7cm},clip,width=0.33\linewidth]{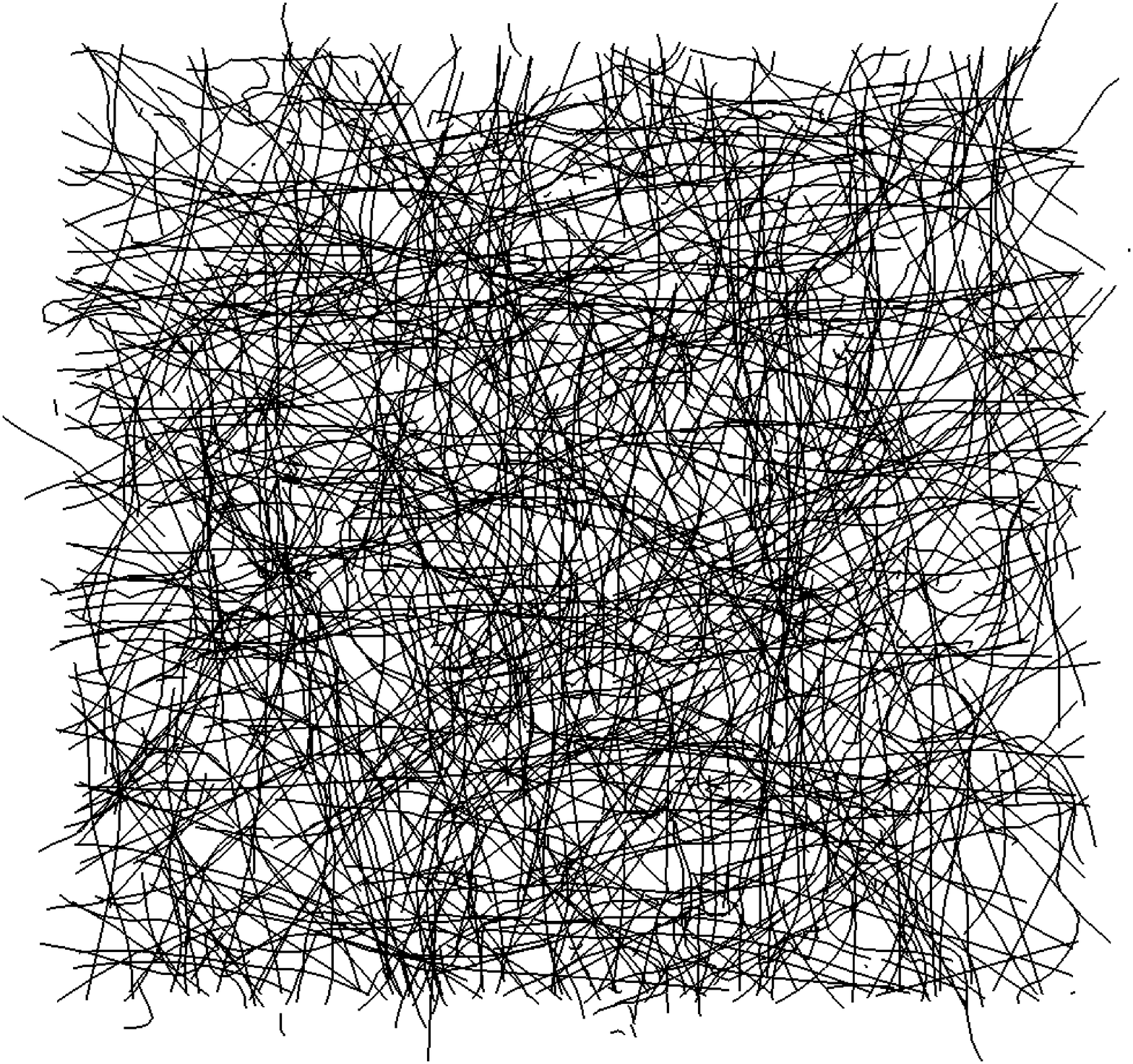}\label{nwIII}}
\caption{{Three examples of networks: a regular square network (a),  a regular square network with randomly perturbated nodes (b), and a fiber network (c) (generated as in \cite{Lic}).}}
\label{ThreeNet}
\end{figure*}

The objective of this paper is to develop a numerical upscaling method for networks, circumventing the computational issues arising when materials of macrosize are considered. The idea is to reduce the size of the system by introducing a subspace $V_\textnormal{ms}\subset V$, as a coarse representation of the network. This space is called the multiscale space and it should fulfil the condition that $\dim V_\textnormal{ms}$ is much lower than $\dim V$. The multiscale solution is attained from the problem
\begin{align*}
\textnormal{Find }u\in V_\textnormal{ms}:\quad v^T Ku_\textnormal{ms}=v^T F, \quad \forall v\in V_\textnormal{ms}.
\end{align*}
The aim is to construct a multiscale space such that an error $\|u-u_\textnormal{ms}\|$ is small. To achieve this, a FE-type coarse space is first introduced, which does not have the desired approximation properties. This coarse space is then modified by solving local sub-networks problems, resulting in the desired multiscale space. In the following section such a numerical homogenization method is presented.

\section{Numerical homogenization of networks}
\label{UpscalingMethod}
Consider a network with $N$ nodes and properties governed by a symmetric and positive semi-definite  matrix $K\in\mathbb{R}^{n\times n}$,  where $n=d\cdot N$ is the number of degrees of freedom of the network, and $d$ denotes the number of degrees of freedom at each node. For instance, for an elastic network, where node displacements are to be solved, the number of degrees of freedom at each node will be two or three, depending on if the network is two- or three-dimensional.  In the following presentation the space is assumed to be two-dimensional, but the method works analogously for three dimensions. Denote by $p_i\in\mathbb{R}^2$ the position of the node corresponding to degree of freedom $i = 1, \dots, n$. Note that groups of $d$ degrees of freedom correspond to the same position.

Let the solution vector be denoted $u\in \mathbb{R}^{n}$. The ordering of nodes and their degrees of freedom is arranged such that if $d=2$, $u(1)$ and $u(2)$ correspond to the first and second degree of freedom of node 1, $u(3)$ and $u(4)$ correspond to the first and second degree of freedom of node 2, and so on, with analogous ordering if $d$ is larger. Here $u(i)$ denotes the $i$:th component of vector $u$. Let $F\in\mathbb{R}^{n}$ denote the load vector.  The system $Ku=F$ is not necessarily solvable without prescribing some degrees of freedom. Consider fixed constraints with zero displacement (non-zero displacement is treated in Section \ref{NonzeroFixation}) and let $\mathcal{N}_D\subset \{1,\dots, n\}$ be the set of indices corresponding to the fixed degrees of freedom. Let $\mathcal{N}=\{1,\dots, n\}\setminus\mathcal{N}_D$. Denote by $V \subset\mathbb{R}^{n}$ the restricted solution space defined by
\begin{align*}
V = \{v\in \mathbb{R}^n : v(i)=0, \quad  i\in \mathcal{N}_D\}.
\end{align*}
The variational formulation of the network displacement problem reads:
\begin{align}
\label{DisplacementProblem}
\begin{split}
\textnormal{Find }u\in V:\quad v^T Ku=v^T F, \quad \forall v\in V.
\end{split}
\end{align}
For the problem to be solvable it is assumed that $K$ in addition to being symmetric, also is positive definite on the restricted solution space $V$.

\subsection{\textbf{Coarse grid representation}}
The overall idea of the upscaling method is to introduce a coarse grid, representing the network at macroscale. See Fig. \ref{ExampleNetworkWithCoarseGridI} for an illustration of a network with a coarse grid representation. At each coarse node, $d$ number of basis functions are defined similarly as in the finite element method. These basis functions span a low dimensional solution space which gives an insufficient description of the fine scale features. To include the fine scale information, the basis functions are modified by solving local sub-network systems. Thereafter the modified basis functions are used to solve a global system, smaller than the full system including all nodes, resulting in an upscaled approximation of the original problem. In what follows, the details of this procedure are described.

\begin{figure*}[ht!]
\centering
\subfigure[Network and FEM quadrilateration.]{\includegraphics[trim={2cm 1cm 2cm 0.5cm},clip,width=0.31\linewidth]{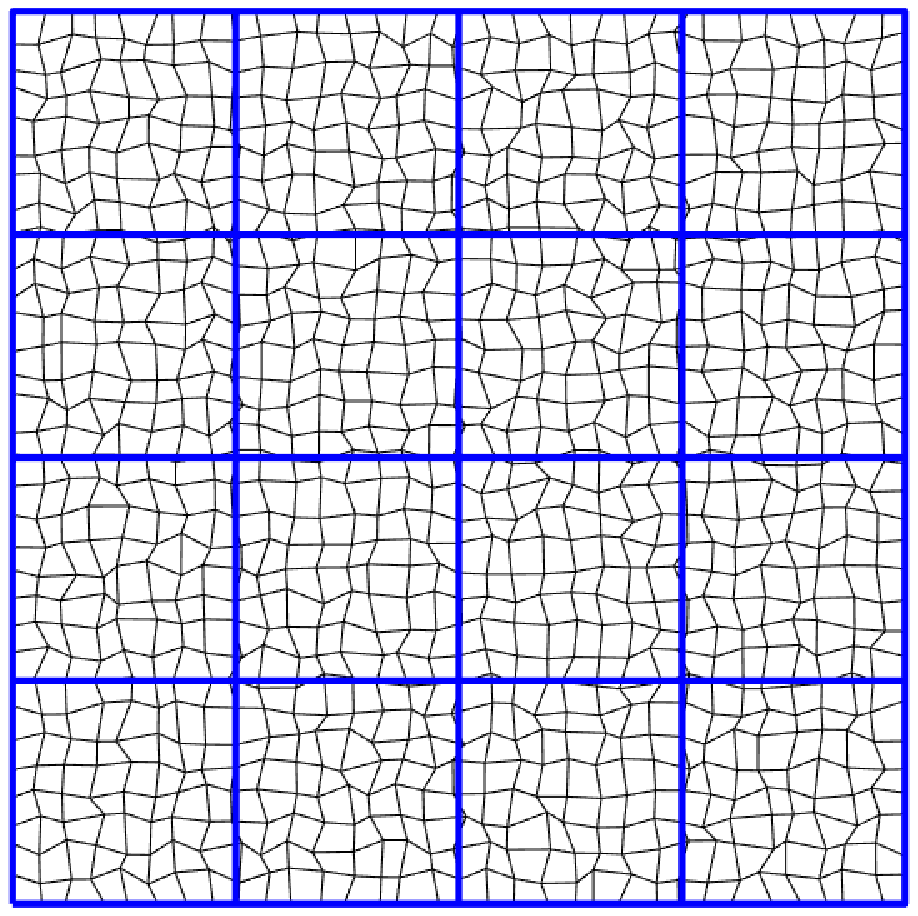} \label{ExampleNetworkWithCoarseGridI}}
\hspace{0.1cm}
\subfigure[A bilinear basis function $\Lambda_i:\mathbb{R}^2\to\mathbb{R}$ of the FEM grid.]{\includegraphics[trim={4cm 1.4cm 3.5cm 3cm},clip,width=0.31\linewidth]{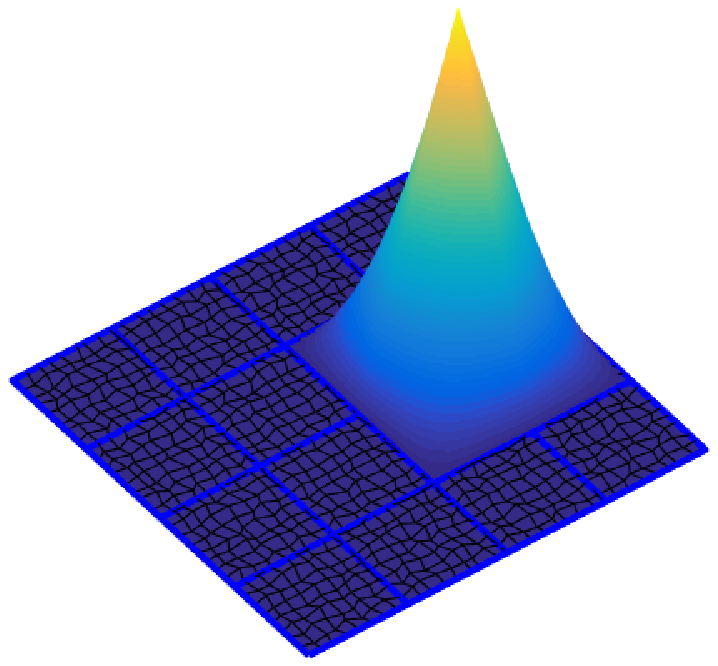} \label{ExampleNetworkWithCoarseGridII}}
\hspace{0.1cm}
\subfigure[The interpolated bilinear basis function $\lambda_i\in V_H\subset\mathbb{R}^n$.]{\includegraphics[trim={4cm 1.4cm 3.5cm 3cm},clip,width=0.31\linewidth]{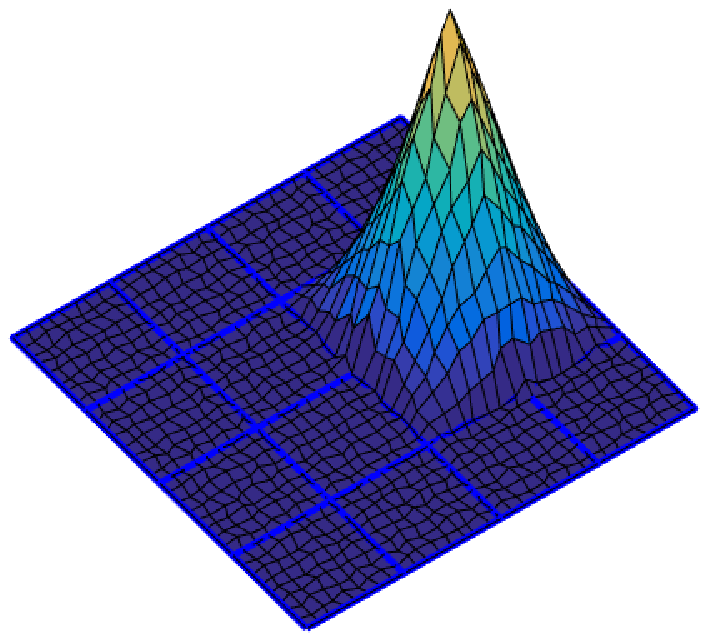} \label{ExampleNetworkWithCoarseGridIII}}
\caption{{Example of a square network with a FEM quadrilateration representation. }}
\label{ExampleNetworkWithCoarseGrid}
\end{figure*}

Let the coarse grid be denoted $\mathcal{T}$, containing $M$ coarse nodes and let $m=d\cdot M$ be the degrees of freedom of the coarse grid. One choice of coarse grid is a quadrilateration as in Fig. \ref{ExampleNetworkWithCoarseGridI}. It is assumed that the coarse grid constitutes a good approximation of the computational domain of the network, and that each coarse element contains at least one network node and that $N> M$. Let $\Lambda_i:\mathbb{R}^2\to\mathbb{R}, \,i=1,\dots, m$, denote the coarse nodal basis functions of the grid $\mathcal{T}$. For a quadrilateration, bilinear basis functions are suitable, illustrated in Fig. \ref{ExampleNetworkWithCoarseGridII}.

Let $ \mathcal{M}_D\subset \{1,\dots, m\}$ be the set of indices corresponding to fixed coarse degrees of freedom. The fixation of coarse grid nodes is determined from the set of fixed network nodes, $\mathcal{N}_D$. Consider a coarse node with basis function $\Lambda_i$, describing for instance the $x$-displacement of that node. If there exists a network node with fixed degree of freedom $j$ such that $p_j$ lies in the support of $\Lambda_i$, and $j$ also describes $x$-displacement, then the coarse degree of freedom $i$ should be fixed. This is illustrated in Fig. \ref{ExampleFixedNodes}. The fixation condition is equivalently stated as:
\begin{align}
\label{FixationCondition}
i\in \mathcal{M}_D \quad\textnormal{if} \quad\exists j\in \mathcal{N}_D  :\quad  \Lambda_i(p_j)\neq 0 \textnormal{ and } i\equiv j \pmod{d}.
\end{align}
For networks with more complex boundary geometry, the coarse grid has to be refined at the boundary to attain a proper representation of the fixed boundary conditions.

\begin{figure*}[ht!]
\centering
\subfigure[Only fixed network nodes encircled.]{\includegraphics[trim={0cm 0cm 0cm 0cm},clip,width=0.3\linewidth]{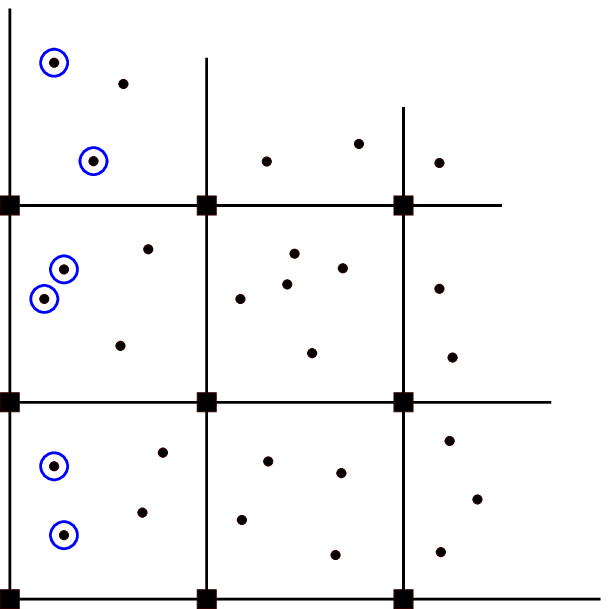} \label{ExampleFixedNodesNetwork}}
\hspace{0.5cm}
\subfigure[Both fixed network nodes and fixed coarse grid nodes encircled.]{\includegraphics[trim={0cm 0cm 0cm 0cm},clip,width=0.3\linewidth]{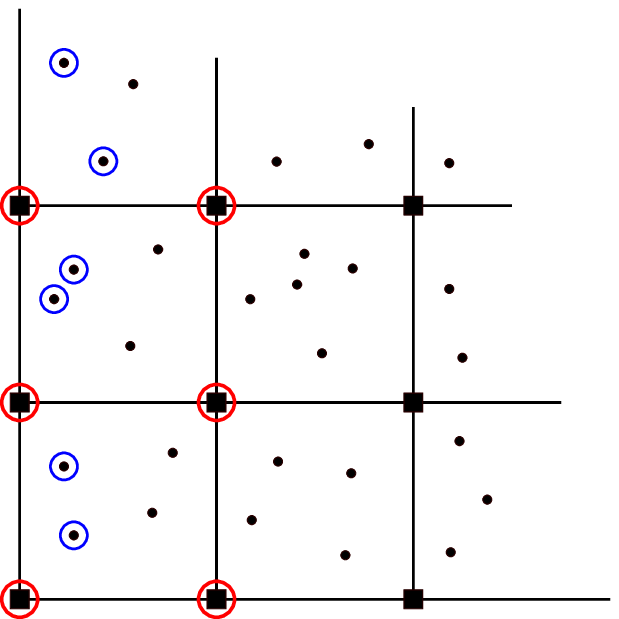} \label{ExampleFixedNodesCoarse}}
\caption{{Example illustrating the fixation of coarse network nodes for $d=1$. The coarse grid nodes are marked with squares and the network nodes with dots. Fixed nodes are encircled. To the left, fixed network nodes are marked with circles, illustrating the set $\mathcal{N}_D$. Based on $\mathcal{N}_D$ and the condition \eqref{FixationCondition}, the coarse grid nodes which should be fixed, $\mathcal{M}_D$, have been marked with larger circles in the right plot. }}
\label{ExampleFixedNodes}
\end{figure*}

Let $\mathcal{M}=\{1, \dots, m\}\setminus \mathcal{M}_D$ denote the set of nonprescribed coarse degrees of freedom. The positions of the coarse nodes, $\{P_i\}_{i=1}^m$, defined similarly as the positions of the network nodes, are a subset of $\mathbb{R}^2$, likewise as the nodes of the network. However, these two subsets do not have to be related, but as already noted, it is assumed that each coarse element contains at least one network node.

Next, two vector spaces are introduced, the coarse space $V_H$, and the detail space $W$. The coarse space is defined from the coarse basis functions in the following way. Let $\lambda_i\in \mathbb{R}^n, i=1,\dots, m$, be the interpolation of the coarse nodal basis functions to the network nodes given by
\begin{align*}
\lambda_i(j) = \begin{cases}
\Lambda( p_j), \,\,\quad \textnormal{if} \quad   i\equiv j \pmod{d},\\
0, \,\quad\quad\quad \textnormal{else}.
\end{cases}
\end{align*}
See Fig. \ref{ExampleNetworkWithCoarseGridIII} for an illustration of the interpolated vector $\lambda_i$ of the coarse nodal basis function $\Lambda_i$.

The coarse space is defined as the span of the interpolated non-fixed basis functions, that is
\begin{align*}
V_H = \textnormal{span}(\{\lambda_i\}_{i\in\mathcal{M}}),
\end{align*}
with dimension $\dim V_H = m_H := |\mathcal{M}|$. Note that $\lambda_i$ is defined for all $i=1,\dots, m$, but $\lambda_i\in V$ only if $i\in\mathcal{M}$. Let the matrix $B_H = [\{\lambda_i\}_{i\in\mathcal{M}}]\in\mathbb{R}^{n\times m_H}$ contain the basis vectors of the coarse space $V_H$ as its columns. It is   assumed that the columns are linearly independent. The matrix $B_H$ is called the prolongation matrix and acts as a map $B_H:\mathbb{R} ^{m_H}\to V_H$.

To define the detail space, a restriction matrix $C_H\in \mathbb{R}^{m_H\times n}$ is introduced acting as a map $C_H:\mathbb{R}^n\to\mathbb{R}^{m_H}$. In this work, the restriction matrix is chosen as $C_H=B_H^T$ (for examples of other choices of  restriction operator, see \cite{LOD2}). With $C_H=B_H^T$, an equivalent definition of the coarse space $V_H$ is as the range of the map $B_HC_H:V\to V_H$, that is
\begin{align*}
V_H= \{B_HC_Hv:\,\, v\in V\}.
\end{align*}
The detail space is defined as the null space of the restriction matrix:
\begin{align*}
W&= \{v\in V: \,\,C_Hv = 0\}.
\end{align*}
With $C_H=B_H^T$ it holds that $v\in W$ if the bilinear weighted average of $v$  is  zero for each interpolated bilinear basis function $\lambda_i$, i.e. $\lambda_i^T v= 0,\,\forall i\in\mathcal{M}$.

 The coarse space and the detail space constitute a splitting of $V$ such that each $v\in V$ can be uniquely decomposed as $v= v_H+w$ where $v_H\in V_H$ and $w\in W$. Before proving this fact, a lemma is stated showing the relation between the spaces $\mathbb{R}^{m_H}$, $\mathbb{R}^{n}$ and $V_H$, and the maps in-between, illustrated in Fig. \ref{MapOverview}.

\begin{figure*}[ht!]
\centering
\subfigure{\includegraphics[trim={0cm 0cm 0cm 0cm},clip,width=0.7\linewidth]{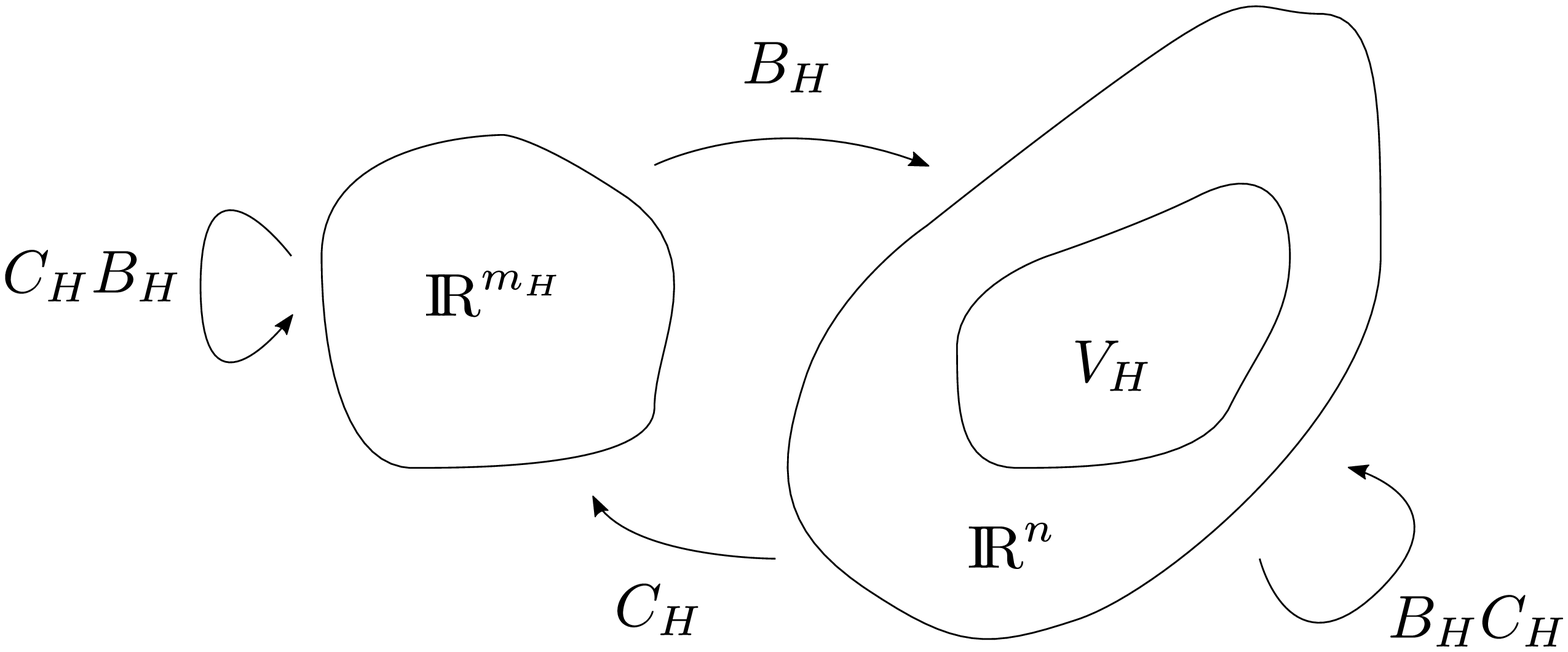}}
\caption{{A sketch of the two spaces $\mathbb{R}^{m_H}$ and $\mathbb{R}^n$, and the subspace $V_H\subset\mathbb{R}^n$. The four mappings $B_H:\mathbb{R}^{m_H}\to V_H$, $C_H:\mathbb{R}^{n}\to \mathbb{R}^{m_H}$, $C_HB_H:\mathbb{R}^{m_H}\to \mathbb{R}^{m_H}$ and $B_HC_H: \mathbb{R}^{n}\to V_H$ are also shown.}}
\label{MapOverview}
\end{figure*}

\begin{lemma}
\label{InversionLemma}
If $B_H$ has linearly independent columns and $C_H=B_H^T$, then for each $v_H\in V_H$ there exists $\bar v_H\in V_H$ such that $B_HC_H\bar v_H = v_H$. Moreover, if $v_1, v_2\in V_H$ such that $B_HC_Hv_1 = B_HC_H v_2$, then $v_1=v_2$.
\end{lemma}
\begin{proof}
The map $B_H:\mathbb{R}^{m_H}\to V_H$ is one-to-one since $B_H$ has linearly independent columns. Moreover, the map $C_HB_H:\mathbb{R}^{m_H}\to \mathbb{R}^{m_H}$ is invertible since $C_HB_H=B_H^TB_H$ is symmetric and positive definite due to the fact that $x^TC_HB_Hx=|B_Hx|^2 \geq 0$ and $|B_Hx|=0$ implies $x=0$. Given $v_H\in V_H$, it exists $a\in\mathbb{R}^{m_H}$ such that $B_Ha=v_H$, and since $C_HB_H$ is invertible it exists $b\in \mathbb{R}^{m_H}$ such that $C_HB_Hb=a$. Therefore $v_H=B_HC_HB_Hb$ leading to $\bar v_H  = B_Hb\in V_H$.

To prove the second part, assume $v_1\neq v_2$. Then there exists $a_1,a_2\in \mathbb{R}^{m_H}$ with $a_1\neq a_2$ such that $v_1=B_Ha_1$ and $v_2=B_Ha_2$. Since $C_HB_H$ is invertible, $C_HB_Ha_1\neq C_HB_Ha_2$, contradicting the fact that $B_HC_HB_Ha_1= B_HC_HB_Ha_2$, hence $v_1=v_2$. \qed
\end{proof}

\begin{proposition}
\label{SplittingCoarse}
If $B_H$ has linearly independent columns and $C_H=B_H^T$, then $V = V_H\oplus W$ uniquely.
\end{proposition}
\begin{proof}
For $v\in V$, let $v_H=B_HC_Hv$. Lemma \ref{InversionLemma} states the existence of $\tilde v_H\in V_H$ such that $v_H=B_HC_H\tilde v_H$. Since $B_H$ has linearly independent columns the relation $0=v_H-v_H=B_HC_Hv-B_HC_H\tilde v_H = B_H(C_Hv-C_H\tilde v_H)$ implies that $C_Hv-C_H\tilde v_H = 0$ with conclusion that $v-\tilde v_H\in W$. Therefore $v=\tilde v_H+(v-\tilde v_H)$ is a desired decomposition. To show uniqueness, consider two decompositions $v=v_{H,1} + w_1$ and $v= v_{H,2}+w_2$. Then $v_{H,1} + w_1= v_{H,2}+w_2$, and applying $B_HC_H$ on both sides gives $B_HC_Hv_{H,1}=B_HC_Hv_{H,2}$. From the last part of Lemma \ref{InversionLemma} it follows that $v_{H,1}=v_{H,2}$. \qed
\end{proof}

Using the detail space $W$, together with the connectivity matrix $K$, the multiscale space $V_\textnormal{ms}$ is defined as the $K$-orthogonal complement of $W$:
\begin{align*}
V_\textnormal{ms} = \{v\in V: w^TKv = 0,\quad\forall w\in W\}.
\end{align*}
The spaces $W$ and $V_\textnormal{ms}$ constitute another splitting of $V$ implying that every $v\in V$ can be decomposed uniquely as $v = v_\textnormal{ms}+w$ where $v_\textnormal{ms}\in V_\textnormal{ms}$ and $w\in W$.

\begin{proposition}
\label{Splitting}
Assume $B_H$ has linearly independent columns and $C_H=B_H^T$. If $K$ is symmetric and positive definite on $V$, then $V = V_\textnormal{ms}\oplus W$ uniquely.
\end{proposition}
\begin{proof}
Consider $v\in V$. From Proposition \ref{SplittingCoarse} it is known that $v=v_H+ \tilde w$ with $v_H\in V_H$ and $\tilde w\in W$. Let $z\in W: \, w^TKz=w^TKv_H,\, \forall w\in W$, which has a unique solution since $K$ is symmetric and positive definite on $V$. Define $v_\textnormal{ms}=v_H-z$ and $w=\tilde w + z$, where the second sum is in $W$. Since $x^TKv_\textnormal{ms}=x^TKv_H - x^TKz=0, \,\forall x\in W$, it is true that $v_\textnormal{ms}\in V_\textnormal{ms}$, giving the desired decomposition as $v = v_\textnormal{ms} + w$. To prove uniqueness, consider $v=v_{\textnormal{ms}, 1} + w_ 1$ and $v=v_{\textnormal{ms}, 2} + w_ 2$. Then $0=x^TK(v-v) = x^TK(w_1 + v_{\textnormal{ms}, 1} - w_2 - v_{\textnormal{ms}, 2})= x^TK(w_1-w_2), \, \forall x\in W$, implying that $w_1=w_2$. \qed
\end{proof}

The multiscale solution, $u_\textnormal{ms}\in V_\textnormal{ms}$, to the original problem \eqref{DisplacementProblem}, is defined as the solution to the problem
\begin{align}
\label{MSProblem}
\begin{split}
\textnormal{Find }u_\textnormal{ms}\in V_\textnormal{ms}:\quad v^TKu_\textnormal{ms} = v^TF,\quad\quad\forall v\in V_\textnormal{ms}.
\end{split}
\end{align}

\begin{proposition}
\label{ExistenceOfSolution}
Let $K$ be symmetric and positive definite on $V$, and $F\in V$. Then there exists a unique solution to problem \eqref{MSProblem}.
\end{proposition}

\begin{proposition}
\label{CorrectionTheorem}
Let  $u_f\in W$ be such that $w^TKu_f=w^TF, \,\, \forall w \in W$. Then the sum $ u  =  u_\textnormal{ms} + u_f$, where $u_\textnormal{ms}$ is the solution to the multiscale problem \eqref{MSProblem}, solves the original problem \eqref{DisplacementProblem}.
\end{proposition}
\begin{proof}
Same arguments as used in Proposition \ref{ExistenceOfSolution} show that there exists a unique such $u_f$. According to Proposition \ref{Splitting}, $v\in V$ can be decomposed as $v=v_\textnormal{ms} + w$, where $v_\textnormal{ms}\in V_\textnormal{ms}$ and $w\in W$. Using orthogonality, and that $u_f$ and $u_\textnormal{ms}$ are solutions to their respective problem, it can be derived that
\begin{align*}
\begin{split}
v^TKu &= (v_\textnormal{ms} + w)^TK(u_\textnormal{ms} + u_f) \\
&= v_\textnormal{ms}^TKu_\textnormal{ms} + v_\textnormal{ms}^TK u_f+ w^T K u_\textnormal{ms}+ w^T Ku_f\\
&= v_\textnormal{ms}^TF +0+ 0+ w^T F\\
&= v^TF.
\end{split}
\end{align*}
\qed
\end{proof}
To solve the multiscale problem \eqref{MSProblem}, it is convenient to construct a basis for the multiscale space $V_\textnormal{ms}$, which can be used to simplify the problem to a matrix equation.  A basis for $V_\textnormal{ms}$ is constructed using the vectors $\lambda_i$, by defining modification vectors $\phi_i\in V$, $i\in\mathcal{M}$, as solutions to the problems
\begin{align}
\label{BasisModification}
\phi_i\in W  : \quad w^TK(\lambda_i-\phi_i) = 0,\quad \forall w\in W.
\end{align}

\begin{proposition}
If $K$ is symmetric and positive definite on $V$, $B_H$ has linearly independent columns and $C_H=B_H^T$,  then the vectors $\{\lambda_i-\phi_i\}_{i\in\mathcal{M}}$ constitute a basis for $V_\textnormal{ms}$.
\end{proposition}
\begin{proof}
The problem to find $\phi_i$ such that $w^TK\phi_i=w^TK\lambda_i$, $\forall w\in W$, has a unique solution $\phi_i\in W$ since $K$ is symmetric and positive definite on $V$. By construction, it is also true that $\lambda_i-\phi_i\in V_\textnormal{ms}$. To prove linear independence, consider $\sum a_i (\lambda_i-\phi_i)=0$ and apply $B_HC_H$ to both sides. Using that $C_H\phi_i=0$ gives $B_HC_H\sum a_i\lambda_i = 0 = B_HC_H0$, and by the second part of Lemma \ref{InversionLemma} it follows that $\sum a_i\lambda_i = 0$ implying $a_i=0$. \qed
\end{proof}

Using the above constructed basis for $V_\textnormal{ms}$, $\{\lambda_i - \phi_i\}_{i\in\mathcal{M}}$, to assemble the matrix
\begin{align*}
B_\textnormal{ms}= [\{\lambda_i - \phi_i\}_{i\in\mathcal{M}}]\in\mathbb{R}^{n\times m_H},
\end{align*}
reduces the variational form of the multiscale problem \eqref{MSProblem} to the equivalent matrix problem
\begin{align}
\label{MatrixMSProblem}
B_\textnormal{ms}^TKB_\textnormal{ms}U_\textnormal{ms} = B_\textnormal{ms}^TF,
\end{align}
where $U_\textnormal{ms}\in \mathbb{R}^{m_H}$ and $u_\textnormal{ms} = B_\textnormal{ms}U_\textnormal{ms}$.

At this point, a multiscale space with low dimension compared to the solution space $V$ has been constructed as was desired in the problem formulation. Moreover, it has been shown that the problem can be solved through a matrix equation after constructing a basis for the multiscale space. However, the problems of solving the modified basis functions have the same size as the original problem. It turns out that this can be circumvented, utilizing that the modifications $\phi_i$ decay exponentially, implying  that the problems can be localized. This is presented in the following section.

\subsection{\textbf{Localization}}
\label{Localization}
The described method  requires systems to be solved which are as large as $K$. However, as will be demonstrated by numerical examples in Sec.~\ref{NumericalExamples}, the modifications $\phi_i$ decay fast away from its coarse node. Therefore the problems \eqref{BasisModification}, of calculating $\phi_i$, can be localized with preserved convergence rates.  The localization is accomplished by solving each problem \eqref{BasisModification} on a restricted domain, called patch.

As in FEM, it is natural to assemble the stiffness matrix elementwise. In this work, it is suitable to assemble the connectivity matrix $K$ over each coarse element $E$, such that $K=\sum_{E\in\mathcal{T}} K_E$. The local element connectivity matrices $K_E:V\to V$ are assembled for each element $E\in\mathcal{T}$ by only considering edges in each element. See Fig. \ref{ElementAssemble} for an illustration. For unstructured networks, edges may intersect the elements. Such a situation is resolved by temporarily dividing the edges at intersection points. Using the decomposition of the connectivity matrix into local element matrices, the modifications $\phi_i$ can analogously be assembled as $ \sum_{E\in\mathcal{T}} \phi_i^E$, where $\phi_i^E$ is the solutions to the problem
\begin{align*}
\textnormal{Find }\phi_i^E\in W: \quad w^TK\phi_i^E=w^TK_E\lambda_i,\quad \forall w\in W.
\end{align*}
\begin{figure*}[ht!]
\centering
\subfigure{\includegraphics[trim={2.9cm 1.4cm 2.5cm 1.1cm},clip,width=0.4\linewidth]{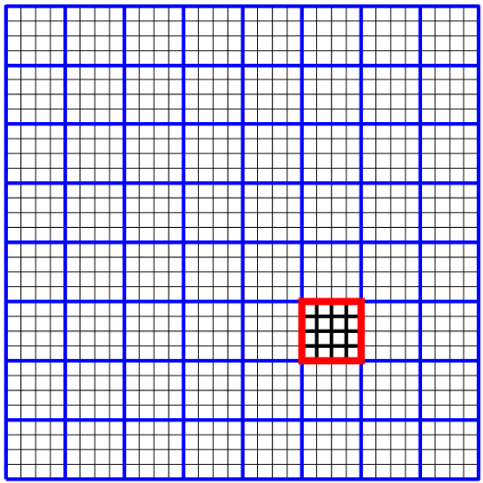}}
\caption{{Element $E$ is marked with a square and the edges which are included in the assemble of $K_E$ is marked with thick lines.}}
\label{ElementAssemble}
\end{figure*}

\begin{proposition}
The sum of the elementwise modifications, $ \sum_{E\in\mathcal{T}} \phi_i^E$, solves the original problem \eqref{BasisModification}.
\end{proposition}
\begin{proof}
It follows that
\begin{align*}
w^TK\sum_{E\in\mathcal{T}} \phi_i^E = \sum_{E\in\mathcal{T}} w^TK \phi_i^E = \sum_{E\in\mathcal{T}} w^TK_E \lambda_i =  w^T\sum_{E\in\mathcal{T}}K_E \lambda_i = w^TK \lambda_i.
\end{align*}
\qed
\end{proof}
With elementwise assembly, it is convenient to use patches centred at each element. For an element $E\in\mathcal{T}$, let its patch be denoted $\omega_E\subset\mathbb{R}^2$. One suitable choice of patch geometry is a circle with center coinciding with the element center, as illustrated in Fig. \ref{PatchExam1a}. Another choice is to use a fixed number of layers of coarse elements surrounding the considered element $E$, as depicted in Fig. \ref{PatchExam1b}. Let the patch size be described by the parameter $\rho$ such that  $\rho H$ is the radius of the patch, where $H$ denotes the coarse element size. In the two examples shown in Fig. \ref{PatchExam1}, the value of $\rho$ corresponds to 1.5. Let $\mathcal{N}_E \subset\mathcal{N}$, denote the degrees of freedom of network nodes that are in the patch $\omega_E$. Similarly, let $\mathcal{M}_E \subset\mathcal{M}$, denote the degrees of freedom of coarse nodes that are in the patch of element $E$.

\begin{figure*}[ht!]
\centering
\subfigure[Circular patch with origin at element center.]{\includegraphics[trim={2.8cm 1cm 2cm 0.8cm},clip,width=0.4\linewidth]{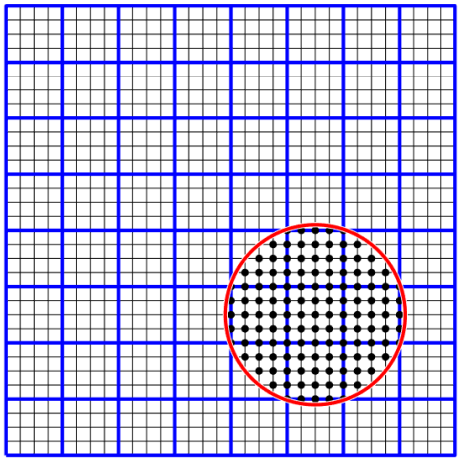}\label{PatchExam1a}}
\hspace{1cm}
\subfigure[Patch with one layer of elements surrounding the element.]{\includegraphics[trim={2.8cm 1cm 2cm 0.8cm},clip,width=0.4\linewidth]{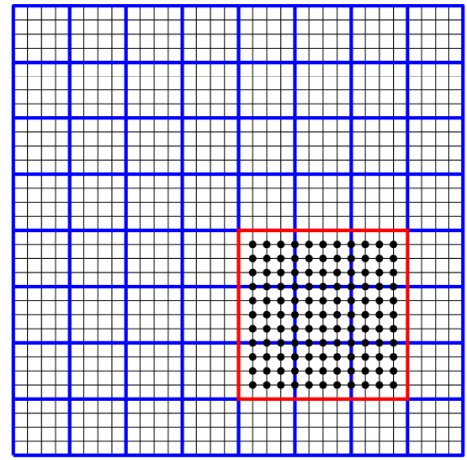}\label{PatchExam1b}}
\caption{{Two square networks with different types of patch geometries.}}
\label{PatchExam1}
\end{figure*}

For each element $E$, its localized subspace of the detail space is defined according to
\begin{align*}
\tilde W_E= \{w\in W: \quad w(i) = 0, \quad \forall i\notin \mathcal{N}_E\}.
\end{align*}
The localized modification $\tilde\phi_i$ is attained by solving the problems
\begin{align*}
\textnormal{Find }\tilde \phi_i^E\in \tilde W_E: \quad w^TK \tilde \phi_i^E=w^TK_E\lambda_i,\quad \forall w\in \tilde W_E,
\end{align*}
for each element $E$, and taking the sum $ \tilde \phi_i = \sum_{E\in\mathcal{T}} \tilde\phi_i^E$. The resulting localized multiscale space is denoted $\tilde V_\textnormal{ms}$, and is defined as
\begin{align*}
\tilde V_\textnormal{ms} =\textnormal{span}(\{\lambda_i-\tilde\phi_i\}_{i\in\mathcal{M}}).
\end{align*}

\subsection{\textbf{Algebraic formulation}}
\label{AlgebraicFormulation}
In this section the algebraic formulation of the numerical multiscale method is presented. The multiscale method consists of two main steps. First,  the basis for the multiscale space $V^\textnormal{ms}$ is constructed by calculating each $\phi_i$ from \eqref{BasisModification}. Secondly, the attained modified basis is used to solve the global multiscale problem \eqref{MSProblem}. In this section, elementwise assembling and localization, described in Sec.~\ref{Localization}, is employed.

First, a useful matrix notation is introduced. Consider a matrix $A\in\mathbb{R}^{a\times b}$. Let $\mathcal{A}\subset\{1,2,\dots,a\}$ and  $\mathcal{B}\subset\{1,2,\dots,b\}$ be subsets of the matrix row and column indices respectively. A new matrix $A(\mathcal{A}, \mathcal{B})\in\mathbb{R}^{|\mathcal{A}|\times|\mathcal{B}|}$ is extracted from $A$ by only considering rows corresponding to  indices in $\mathcal{A}$ and columns corresponding to indices in $\mathcal{B}$, that is $A(\mathcal{A}, \mathcal{B})= (a_{ij})_{(i,j)\in \mathcal{A}\times\mathcal{B}}$ .

Using the introduced matrix notation, the following matrices and vectors are defined for $E\in\mathcal{T}$ and $i\in\mathcal{M}$:
\begin{align*}
\begin{split}
K^E &= K(\mathcal{N}_E, \mathcal{N}_E),\\
C_H^E &= C_H(\mathcal{M}_E, \mathcal{N}_E),\\
r^E_i &= K_E(\mathcal{N}_E, :) \lambda_i.
\end{split}
\end{align*}

The multiscale method can now be formulated as solving several matrix systems. For each element $E\in\mathcal{T}$, and each coarse degree of freedom $i\in\mathcal{M}$, the following system is solved
\begin{align*}
\begin{bmatrix}
K^E & {C_H^E}^T\\
C_H^E & 0
        \end{bmatrix}
\begin{bmatrix}
\tilde\varphi_i^E \\
\eta_i^E
\end{bmatrix} =
\begin{bmatrix}
r_ i^E \\
0
\end{bmatrix} .
\end{align*}
The correction vectors $\tilde\phi_i^E$ is attained from $\tilde\varphi_i^E$ as
\begin{align*}
\begin{split}
\tilde\phi_i^E(\mathcal{N}_E) &= \tilde\varphi_i^E, \\
\tilde\phi_i^E(\mathcal{N}_E^C ) &= 0,
\end{split}
\end{align*}
where $\mathcal{N}_E^C=\{1,\dots,n\}\setminus\mathcal{N}_E$. The full modifications are thereafter calculated according to
\begin{align*}
\tilde\phi_i = \sum_{E\in\mathcal{T}} \tilde\phi_i^E,\quad i\in\mathcal{M},
\end{align*}
and used to assemble the modified basis matrix
\begin{align*}
\tilde B_M=  [\{\lambda_i - \tilde\phi_i\}_{i\in\mathcal{M}}].
\end{align*}
Finally, the following localized version of the global problem \eqref{MatrixMSProblem} is solved:
\begin{align}
\label{MatrixMSProblemLocal}
\tilde B_M^TK\tilde B_M \tilde U_\textnormal{ms} = \tilde B_M^TF,
\end{align}
where the fine multiscale solution is calculated as $\tilde u_\textnormal{ms} = \tilde B_M\tilde U_\textnormal{ms}$.

Summarized, the algebraic formulation of the numerical multiscale method is:
\begin{empheq}[box=\fbox]{align}
\begin{split}
&\textnormal{1. For each } i\in \mathcal{M}  \textnormal{ and } E\in\mathcal{T} \textnormal{ solve} \\
&\quad\quad\quad\quad\begin{bmatrix}
K^E & {C_H^E}^T\\
C_H^E & 0
        \end{bmatrix}
\begin{bmatrix}
\tilde\varphi_i^E \\
\eta_i^E
\end{bmatrix} =
\begin{bmatrix}
r_ i^E \\
0
\end{bmatrix}.\quad\quad\quad \\
&\textnormal{2. Assemble   } \\
&\quad\quad\quad\quad\tilde\phi_i = \sum_{E\in\mathcal{T}} \tilde\phi_i^E,\quad i\in\mathcal{M}.\\
&\textnormal{3. Solve    }\\
&\quad\quad\quad\quad\tilde B_M^TK\tilde B_M \tilde U_\textnormal{ms} = \tilde B_M^TF.\\
&\textnormal{4. Calculate    }\\
&\quad\quad\quad\quad\tilde u_\textnormal{ms} = \tilde B_M\tilde U_\textnormal{ms}.
\end{split}
\end{empheq}

\subsection{\textbf{Non-zero fixed boundary conditions}}
\label{NonzeroFixation}
Consider a displacement problem where $F=0$ and the set of prescribed degrees of freedom, $\mathcal{N}_D$, includes degrees with non-zero  displacement. Let
\begin{align*}
V_D = \{u\in V: u_i=g_i, \quad  i\in \mathcal{N}_D\},
\end{align*}
 where $g_i$ are the prescribed displacement of degree of freedom $i\in\mathcal{N}_D$. Let still $V = \{u\in V: u_i=0, \quad  i\in \mathcal{N}_D\}$. The displacement problem is
\begin{align}
\label{dispProb}
\textnormal{Find }  u\in V_D: \quad v^T Ku=0, \quad \forall v\in V.
\end{align}
Assume $u=u_{0}+g_H$ where $u_{0}\in V$ and $g_H$ for simplicity is a linear combination of $\lambda_i, \,i=1,\dots, m$. The problem \eqref{dispProb} is then equivalent to
\begin{align}
\label{dispProbReformulated}
 \textnormal{Find } u_{0}\in V: \quad v^TKu_{0} = - v^TKg_H, \quad\forall v\in V.
\end{align}
The problem is in this way transformed back to the previous formulation and can be solved with the same method.

Proposition \eqref{CorrectionTheorem} states that the exact solution of the displacement problem is $u=u_\textnormal{ms} + u_f$. For the case with non-zero prescribed displacements and $g_H\in \textnormal{span}(\{\lambda_i\}_{i=1}^m)$, which was described above, the correction vector $u_f$ turns out to be a linear combination of the vectors $\phi_i, \,i= 1,\dots, m$, as the following derivation shows. Consider the non-homogeneous displacement problem \eqref{dispProbReformulated}. The correction is attained from
\begin{align}
\label{CorrectionDispProb}
\textnormal{Find } u_f \in W: \quad w^TKu_f = -w^TKg_H, \quad \forall w\in W,
\end{align}
where $g_H$ was assumed to be a linear combination of $\lambda_i$, i.e.:
\begin{align}
\label{LinearComb}
g_H=\sum_{i=1}^m\alpha_i\lambda_i.
\end{align}
The modification vectors $\phi_i, \,i=1,\dots, m$, is attained from the problems
\begin{align}
\label{ModificationProblemCorrection}
\textnormal{Find } \phi_i \in W: \quad w^TK\phi_i = w^TK\lambda_i, \quad \forall w\in W.
\end{align}
Note that the problems are solved for all $i=1,\dots, m$, compared to before, when only $i\in\mathcal{M}$ were considered. This is necessary to construct the correction $u_f$ as will  be seen next.

Inserting \eqref{LinearComb} into \eqref{CorrectionDispProb} and using \eqref{ModificationProblemCorrection} gives
\begin{align*}
w^TKu_f &= -w^TKg_H=-\sum_{i=1}^m \alpha_i w^TK\lambda_i=-\sum_{i=1}^m \alpha_i w^TK\phi_i=w^TK(-\sum_{i=1}^m \alpha_i \phi_i),
\end{align*}
implying that the correction is the sum
\begin{align*}
u_f = -\sum_{i=1}^m\alpha_i\phi_i.
\end{align*}
For general fixed displacements $g_H$, not necessarily in $V_H$, see the work \cite{HM2014}.

\section{Error analysis}
\label{ErrorAnalysis}
This paper concerns  a quite general network model described by a connectivity matrix $K$ and a right hand side load $F$. In this section, some error bounds are shown. Because of the generality of $K$,  assumptions are needed. It is  assumed that the coarse grid is quasi-uniform finite element mesh with mesh parameter $H\approx m^{-1/2}$. The following two norms on the space $V$ are introduced:
\begin{align*}
\vertiii{v}&:=(v^TKv)^{1/2},\\
\|v\|&:=(v^Tv)^{1/2}.
\end{align*}
Since $K$ is symmetric and positive definite on $V$ the smallest eigenvalue of $K$, denoted $\nu$, fulfils
\begin{align*}
\frac{\vertiii{v}^2}{\|v\|^2}\geq \nu.
\end{align*}
It is assumed that $\nu$ is bounded from below by a constant independent of $n$. For the finite element method posed on a quasi-uniform mesh these definitions and assumptions correspond to the energy norm, the $L^2$-norm and a Poincar\'{e} inequality. A bilinear weighted interpolant  $\pi_H :V\rightarrow V_H$ is defined as  $\pi_H v=\sum_{i\in\mathcal{M}}(\lambda_i^T v)\lambda_i$. Assume the following  error bound:
\begin{align}
\label{AssumedErrorBound}
\|v-\pi_H v\|\leq C H\vertiii{F},
\end{align}
which is expected to hold for slowly varying $F$.

The main source of error in the proposed method is the localization of the multiscale basis functions to patches. To isolate this contribution the vector $\pi_H F \in V_H$ is introduced. Consider the modified model problem: find $\hat{u}\in V$ such that
\begin{align*}
v^T K \hat{u}=v^T \pi_H F,\,\,\,\forall v\in V.
\end{align*}
Given the assumptions, it is  noted that
\begin{align*}
\vertiii{u-\hat{u}}^2 =(u-\hat u)^T (F-\pi_H F)\leq \|u-\hat u\| \|F-\pi_H F\|\leq C\vertiii{ u-\hat{u}}\|F-\pi_H F\|,
\end{align*}
which together with \eqref{AssumedErrorBound} implies
\begin{equation}
\label{eq:dataerror}
\vertiii{ u-\hat{u}}\leq C\|F-\pi_H F\| \leq CH\vertiii{F}.
\end{equation}
This error is viewed as acceptable and  without loss of generality it is assumed that $F\in V_H$.

First it is  shown that the multiscale solution $U_\text{ms}$ is equal to $u$ if $F\in V_H$.
\begin{proposition}\label{FHequivalence}
Let $u$ and $u_\textnormal{ms}$ be defined according to
\begin{align*}
u\in V: \, v^TKu&=v^TF, \,\,\,\forall v\in V,\\
u_\textnormal{ms}\in V_\textnormal{ms}: \, v^TKu_\textnormal{ms}&=v^TF, \,\,\,\forall v\in V_\textnormal{ms}.
\end{align*}
 If  $F\in V_H$, then it holds that $u=u_\textnormal{ms}$.
\end{proposition}
\begin{proof}
Since $F\in V_H$, there exists an $\bar{F}\in V$ such that $B_H C_H\bar F= F$. To show $v^TK u_\textnormal{ms}=v^T F, \, \forall v\in V$, it is noted that it  holds for all $v\in V_{\text{ms}}$ from the definition of $u_\textnormal{ms}$. Due to the splitting $V=V_\textnormal{ms}\oplus W$ it is enough to show that it also is true for all test functions $v_f\in W$. For any $v_f\in W$ it holds that
\begin{align*}
v_f^T F-v_f^T K u_\textnormal{ms}=v_f^T B_H C_H \bar F=(\bar F^TC_H^T B_H^T v_f)^T=0,
\end{align*} 
since $v_f^T K u_\textnormal{ms}=0$ by orthogonality and $B_H^T v_f=C_H v_f=0$ by the definition of the space $W$. Therefore $u_\textnormal{ms}$ solves the same equation as $u$, and due to  uniqueness $u_\textnormal{ms}=u$. \qed
\end{proof}

The error committed by localization is difficult to study without stronger assumptions on the connectivity matrix $K$. It has however been analysed for several concrete cases, for instance when $K$ is the stiffness matrix arising from discretizing the Poisson equation \cite{LOD1} and the elasticity equations \cite{HePe2016} with the finite element method. For these cases it is true that
\begin{align}
\label{eq:decay}
\max_{w\in V_\textnormal{ms}}\min_{v\in \tilde V_\textnormal{ms}}\frac{\vertiii{w-v}}{\vertiii{w}}\leq Ce^{-c\rho},
\end{align}
where $C$ and $c$ are constants independent of $H$. In Sec.~\ref{NumericalExamples} it is shown, through numerical validation, that this relation is true for more complicated network models. The theoretical analysis of this result for the general network model proposed in Sec.~\ref{NetworkModel} is complicated and postponed to future work. Assuming this result gives the following error bound.

\begin{theorem}
Assuming equation \eqref{eq:decay}, the following error bound is true for any load vector $F\in V$:
\begin{align*}
\vertiii{u-\tilde u_\textnormal{ms}}\leq C\|F-\pi_H F\|+Ce^{-c\rho}\|F\|.
\end{align*}
\end{theorem}
\begin{proof}
First consider the case $F\in V_H$. The vectors $ u_\textnormal{ms}\in V_\textnormal{ms}$ and $\tilde u_\textnormal{ms} \in \tilde V_\textnormal{ms}$ solve
\begin{align*}
v^TKu_\textnormal{ms}&=v^T F,\quad\forall v\in V_\textnormal{ms},\\
v^TK\tilde u_\textnormal{ms}&=v^T F,\quad\forall v\in \tilde{V}_\textnormal{ms}.
\end{align*}
Since $\tilde V_\textnormal{ms}\subset V$, Galerkin orthogonality gives 
\begin{align*}
\vertiii{u-\tilde u_\textnormal{ms}}\leq \min_{v\in \tilde V_\textnormal{ms}} \vertiii{u-v},
\end{align*}
Moreover, since $F\in V_H$ it holds that $u=u_\textnormal{ms}$. Using this together with the assumption \eqref{eq:decay} with $w=u=u_\textnormal{ms}$ leads to
\begin{align*}
\vertiii{u_\textnormal{ms}-\tilde u_\textnormal{ms}}\leq \min_{v\in \tilde{V}_{\text{ms}}}\vertiii{u_\textnormal{ms}-v}\leq Ce^{-c\rho}\vertiii{u_\textnormal{ms}}.
\end{align*}
This together with the fact that 
\begin{align*}
\vertiii{u_\textnormal{ms}}^2 = |u_\textnormal{ms}^TK u_\textnormal{ms}| = |u_\textnormal{ms}^TF|\leq  \|u_\textnormal{ms}\| \|F\| \leq C\|F\| \vertiii{u_\textnormal{ms}}
\end{align*}
 gives
\begin{align}
\label{ineqTheoremSub}
\vertiii{u_\textnormal{ms}-\tilde u_\textnormal{ms}}\leq Ce^{-c\rho}\|F\|.
\end{align}
For the general case $F\in V$,  let $\hat u, \hat u_\textnormal{ms}$ and $\hat{\tilde u}_\textnormal{ms}$ denote the different solutions to the problems with load vector $\pi_H F$. Using $\hat u = \hat u_\textnormal{ms}$ and  the inequalities \eqref{eq:dataerror} and \eqref{ineqTheoremSub} finally result in
\begin{align*}
\begin{split}
\vertiii{u-\tilde u_\textnormal{ms}} &= \vertiii{u-\hat u + \hat u_\textnormal{ms} - \hat{\tilde u}_\textnormal{ms} +  \hat{\tilde u}_\textnormal{ms} - \tilde u_\textnormal{ms}} \\
&\leq  \vertiii{u-\hat u} + \vertiii{\hat u_\textnormal{ms} - \hat{\tilde u}_\textnormal{ms}} +  \vertiii{\hat{\tilde u}_\textnormal{ms} - \tilde u_\textnormal{ms}}\\
&\leq C\|F-\pi_H F\|+Ce^{-c\rho}\|\pi_H F\| + C\|F-\pi_H F\| \\
&\leq C\|F-\pi_H F\|+ Ce^{-c\rho}\left(\|F-\pi_H F\| + \| F\|\right)\\
&\leq C\|F-\pi_H F\|+Ce^{-c\rho}\|F\|.
\end{split}
\end{align*}
\qed
\end{proof}

\section{Network model for paper-based materials}
\label{NetworkModel}
In this section, a two-dimensional elasticity network model is presented, which can be used to model fiber networks in paper-based materials. An elasticity network consists of nodes and edges. When the nodes are displaced, internal forces act to restore the displacements.  These forces act at two types of elements, either on edges or on edge pairs (two edges connected at a joint node). Three types of internal forces are included in this model, one type is related to edges, and two types are related to edge pairs. The model is two-dimensional, static and assumes small deformations.

The network mechanics are governed by force equilibrium equations  assembled at each node. The general form of the equation for each node $i$ reads
\begin{align*}
F^\textnormal{Internal}_i  + F^\textnormal{External}_i = 0,
\end{align*}
where $ F^\textnormal{External}_i\in\mathbb{R}^2$ are all externally applied forces. The internal force $F^\textnormal{Internal}_i\in\mathbb{R}^2$ is, as  mentioned, a sum of three contributions:
\begin{align*}
F^\textnormal{Internal}_i  = F^\textnormal{I}_i  + F^\textnormal{II}_i +  F^\textnormal{III}_i .
\end{align*}
The first force contribution is related to the edges of the network and acts to compensate for changes in length of the edges.  The second force contribution acts on edge pairs  to compensate changes in the angle between the edges of each edge pair.
The third force contribution is included to model the Poisson effect. It acts on edge pairs by introducing a resistance to changes in the total length of the two edges of the pair. In the following sections, the three forces are described. Preparatory some nomenclature is introduced.

Let the network consist of $N$ nodes. Let $(i, j)$ denote the edge connecting node $i$ to $j$ and let $\mathcal{E}$ denote the set of all edges. Note that $(i, j) = (j, i)$. Edge pairs are denoted by $(i, j, l)$ where $j$ is the central node. Denote by $\mathcal{P}$ the set of all edge pairs. Note that $(i,j,l)=(l,j,i)$. Each node $i$ has two degrees of freedom, the $x$-directed displacement and the $y$-directed displacement, contained in the vector $\delta_i\in\mathbb{R}^2$.

All force equilibrium equations can be assembled into a system of the form $-Ku + F = 0$ where $K\in\mathbb{R}^{n\times n}$ is called the elasticity matrix, $u\in\mathbb{R}^n$ is the node displacements, and $F\in\mathbb{R}^n$ is the external forces. The elasticity matrix $K$ is attained by summation of  matrices assembled at edges and edge pairs. The node displacement vector $u$ is arranged according to
\begin{align*}
\begin{bmatrix}
u(2i-1) \\
u(2i)
\end{bmatrix} =
\delta_i , \quad\quad 1 \leq i \leq N,
\end{align*}
and the elasticity matrix is assembled such that the force equilibrium equation for node $i$ is at row $2i-1$ for the $x$-component, and at row $2i$ for the $y$-component.

Let the length of edge $(i, j)$ be denoted $L_{ij}$ and assume that the edge has a width $w_{ij}$. All edges is assumed to have a uniform thickness $z$ in the direction into the plane. The direction vector, $d_{ij}^a$, of an edge $(i,j)$, with respect to node $a\in\{i,j\}$, is defined as
\begin{align*}
d_{ij}^a = \frac{p_b - p_a}{|p_b-p_a|},\quad b \in \{i,j\}, \,\, b\neq a.
\end{align*}
The length change of edge $(i,j)$ is denoted $\Delta L_{ij}$ and given by
\begin{align*}
\Delta L_{ij} = \left(\delta_j-\delta_i\right)\cdot d_{ij}^i.
\end{align*}
In Fig. \ref{NodesAndEdges} some of the introduced notation is depicted. The length change $\Delta L_{ij}$ of an edge, which is not exact but approximated by taking the dot product of the displacement difference onto the direction vector, is illustrated.

\begin{figure*}[ht!]
\centering
\subfigure[]{\includegraphics[trim={0cm 0cm 0cm 0cm},clip,width=0.35\linewidth]{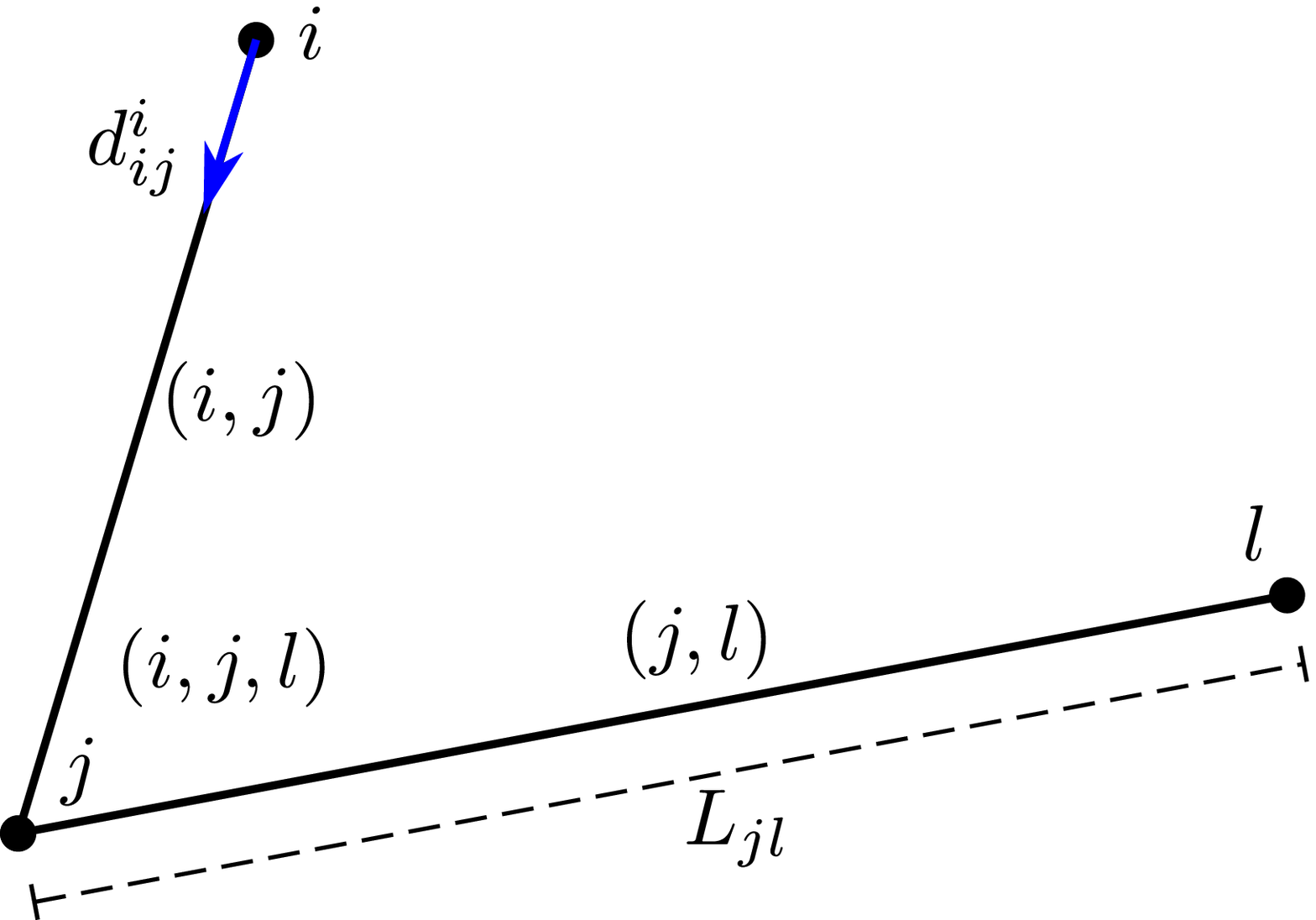}\label{EdgePairPic}}
\hspace{0.3cm}
\subfigure[]{\includegraphics[trim={0cm 0cm 0cm 0cm},clip,width=0.6\linewidth]{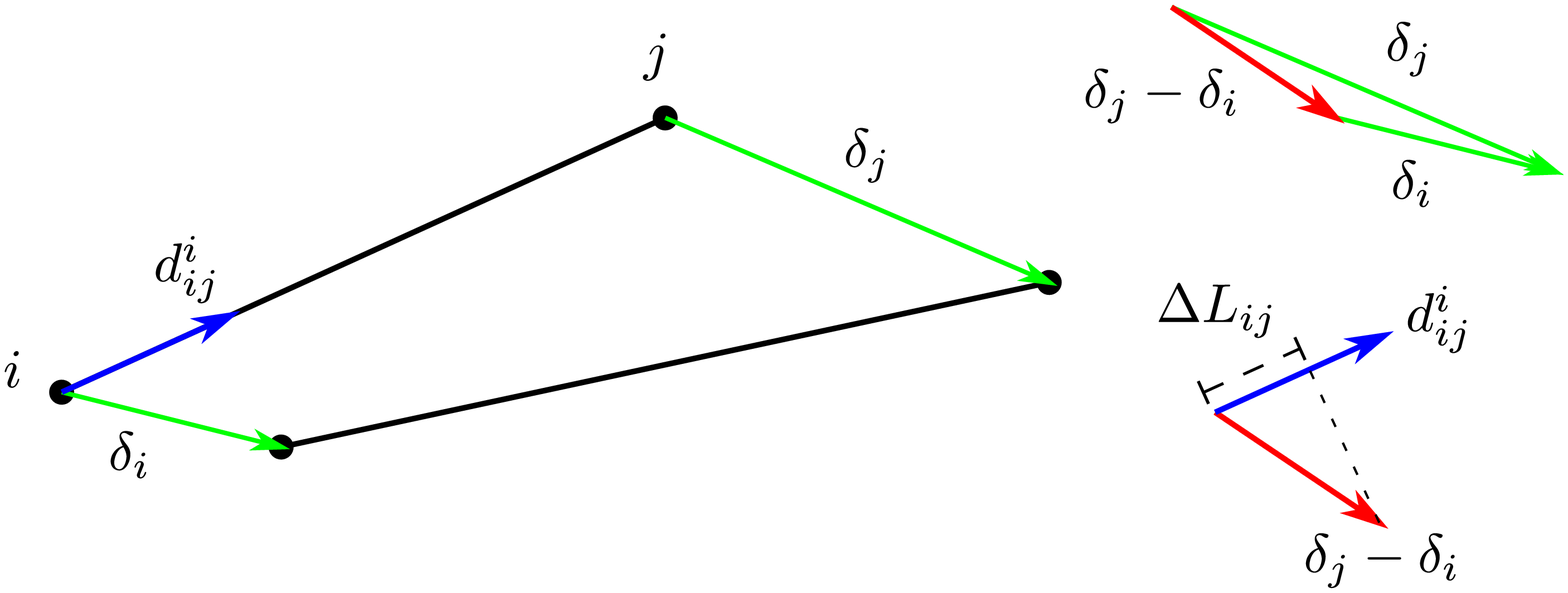}\label{EdgeDisplacement}}
\caption{{Sketches showing the notation used for the network nodes, edges and edge pairs. To the right it is illustrated how the approximate length change $\Delta L_{ij}$ of an edge $(i,j)$ is calculated.}}
\label{NodesAndEdges}
\end{figure*}

\subsection{\textbf{Extension of edges}}
The first force contribution acts at edges due to their internal resistance to length change. When the nodes of an edge are displaced so that the projection of the difference of the node displacements onto the initial edge direction is nonzero, anti-parallel forces arise at the nodes of the edge to restore the length. The tendency of an edge to restore its length is described by the elastic modulus $k_{ij}$.

Consider an edge $(i,j)$ as shown in Fig. \ref{EdgeDisplacement}. When the nodes are displaced, the edge $(i,j)$ will give rise to two forces $F_a^\textnormal{I}(i,j)$, $a\in\{i,j\}$,  acting on node $a$ according to
\begin{align*}
F_a^\textnormal{I}(i,j) = k_{ij} \frac{w_{ij}z}{L_{ij}}\Delta L_{ij}d_{ij}^a, \quad a\in \{i,j\}.
\end{align*}

\subsection{\textbf{Angular deviations of edge pairs}}
The second force contribution acts at edge pairs from their internal tendency to resist change of the angle between their two edges.  When a change in angle occurs, two torques arise at the connecting node acting on one edge each to restore the change. By transforming these torques to force couples the effect can be converted into the force equilibrium equations.

Consider an edge pair $(i, j, l)$ as depicted in Fig. \ref{EdgePairPic}. When the nodes are displaced, an angular change ${\Delta\theta}_{ijl}$ occurs, giving rise to two torques
\begin{align*}
\begin{split}
\tau_i &= \kappa_{ijl} V_{ijl} {\Delta\theta}_{ijl} \hat z,\\
\tau_l &= - \tau_i,
\end{split}
\end{align*}
acting on edge $(i,j)$ and $(j,l)$ respectively, at the position of node $j$. The angular change is a sum of two contributions according to
\begin{align*}
{\Delta\theta}_{ijl} = \delta\theta_{ji} + \delta\theta_{jl}.
\end{align*}
Each term, $\delta_{ji}$ and $\delta_{jl}$, is the angle deviation of respective edge from its initial orientation, as can be seen in Fig. \ref{AngleChange}.
By using the assumption $\alpha \approx \tan\alpha$, the angles $\delta\theta_{ja}, a\in\{i,l\}$ can be calculated according to
\begin{align*}
\delta\theta_{ja} \approx \tan\theta_{ja} = \frac{(\delta_{a}-\delta_j)\cdot n^j_{ja}}{L_{ja}}, \quad a\in\{i,l\},
\end{align*}
where the edge normals are calculated according to
\begin{align*}
\begin{split}
n_{ji}^j &= d^j_{ji} \times \hat z, \\
n_{jl}^j &= -d^j_{jl} \times \hat z.
\end{split}
\end{align*}
Transforming the torques to force couples gives the resulting three forces $F^\textnormal{II}_a(i, j, l), a\in\{i,j,l\}$ from edge pair $(i,j,l)$, acting on node $a$, according to
\begin{align*}
\begin{split}
F^\textnormal{II}_a(i,j, l) &=  -\frac{\kappa_{ijl} V_{ijl} {\Delta\theta}_{ijl}}{L_{aj}}n^j_{ja}, \quad a\in\{i,l\},\\
F_j^\textnormal{II}(i,j,l) &= - F^\textnormal{II}_i(i,j, l) - F^\textnormal{II}_l (i,j, l).
\end{split}
\end{align*}
These forces are illustrated in Fig. \ref{AngleChange}.
\begin{figure*}[ht!]
\centering
\subfigure{\includegraphics[trim={0cm 0cm 0cm 0cm},clip,width=0.5\linewidth]{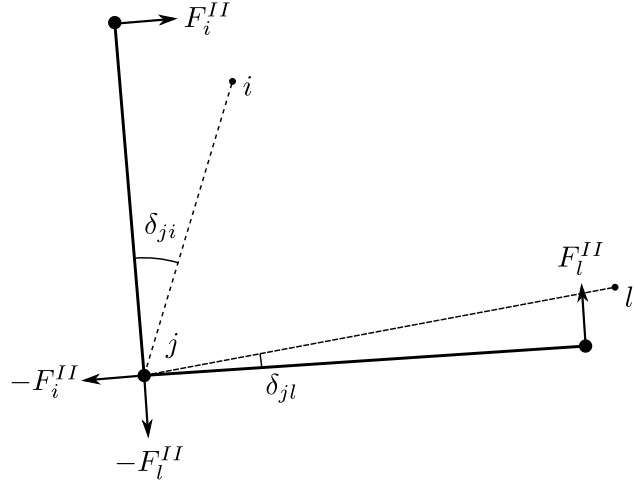}}
\caption{{The initial position of the edge pair $(i,j,l)$ is shown with dashed lines. After displacement, forces act at the nodes of the edge pair to restore the angular change between the edges. The angle deviation of each edge is denoted $\delta_{ji}$ and $\delta_{jl}$, respectively.}}
\label{AngleChange}
\end{figure*}

\subsection{\textbf{Poisson effect of edge pairs}}
The third force contribution results in an effect similar to the Poisson effect and acts at edge pairs. The idea is to add forces that work to keep the total length of the two edges of the pair constant. Hence, when one edge changes length, two kind of forces occur, on one hand forces acting to restore the length of the specific edge, on the other hand forces acting to change the length of the other edge in the pair.

Consider an edge pair $(i,j,l)$, as shown in Fig. \ref{EdgePairPic}. The forces acting at the outer nodes $a\in\{i,l\}$ will be
\begin{align*}
F^\textnormal{III}_a(i,j,l) = -\eta_{ijl} \frac{w_{aj}z}{L_{aj}}\left(\Delta L_{aj}+ \gamma_{ijl}\frac{w_{bj}}{2}\frac{\Delta L_{bj}}{L_{bj}}|n^j_{aj}\cdot d^j_{bj}|\right)d^j_{aj}, \quad a,b\in\{i,l\}, \,\, b\neq a,
\end{align*}
and at the central node
\begin{align*}
F_j^\textnormal{III}(i,j,l) = - F^\textnormal{III}_i(i,j,l)  - F^\textnormal{III}_l(i,j,l).
\end{align*}

\subsection{\textbf {Assembling of elasticity matrix}}
\label{AssemblingMatrix}
The governing equation $-Ku+F=0$ contains all node equilibrium equations as a matrix system. Since each edge and each edge pair leads to separate force contributions, the total elasticity matrix $K$ can be assembled from separate element matrices for each of the three force contributions. Let $K^\textnormal{I}_{ij}, K^\textnormal{II}_{ijl}, K^\textnormal{III}_{ijl}\in\mathbb{R}^{n\times n}$ denote the matrices assembled from the first, second and third force contribution respectively, at different elements (edges $(i,j)$ or edge pairs $(i,j,l)$). These matrices are sparse and the only nonzero elements are defined by the relations
\begin{align*}
\begin{split}
K^\textnormal{I}_{ij}(\{2a-1, 2a\}, \{1, \dots, n\}) u&= F_a^\textnormal{I}(i,j), \quad\quad\,  a\in\{i,j\}, \\
K^\textnormal{II}_{ijl}(\{2a-1, 2a\}, \{1, \dots, n\}) u &= F_a^\textnormal{II}(i,j, l), \quad\,  a\in\{i,j, l\},\\
K^\textnormal{III}_{ijl}(\{2a-1, 2a\}, \{1, \dots, n\}) u &= F_a^\textnormal{III}(i,j, l), \quad  a\in\{i,j, l\}.
\end{split}
\end{align*}
The elasticity matrix $K$ is assembled according to
\begin{align*}
K = &-\sum_{(i,j)\in \mathcal{E}} K^\textnormal{I}_{ij}  - \sum_{(i,j,l)\in\mathcal{P}}\left( K^\textnormal{II}_{ijl} + K^\textnormal{III}_{ijl}\right) .
\end{align*}
The matrix $K$ is symmetric and semi-positive definite. With proper fixation of nodes, the matrix will be positive definite on the restricted solution space. Moreover, for a regular network with uniform coefficients $k$, $\kappa$, $\eta$ and $\gamma$, the presented model is equivalent to the finite difference discretization of the two-dimensional linear elasticity equations.

\section{Numerical results}
\label{NumericalExamples}
In this section, two network problems are solved using the proposed multiscale method and the fiber network model. Both problems are similar, with different boundary conditions and load vectors. Consider a unit square network with nodes and edges in a regular grid pattern, as shown in Fig. \ref{NumExamA}. Let the number of nodes be $(r+1)^2$. In the examples $r=2^7=128$ will be used.  Let $l_i$ and $\mu_i$ represent the standard Lam\'e parameters, with one value for each node $i$. Set the parameters of the network model to
\begin{align*}
\begin{split}
k_{ij} = \frac{2\bar\mu_{ij} + \bar l_{ij}} {5c}, \quad
\kappa_{ij} =  \frac{\bar\mu_{ij} }{4c^2}, \quad
\eta_{ijl} = \frac{2\mu_{j} + l_{j}} {5c},\quad
\gamma_{ijl} =  \frac{2 l_j} {4\eta_{ijl}c^2},\quad
c = \frac{1}{2},
\end{split}
\end{align*}
where $\bar \mu_{ij} = \frac{\mu_i + \mu_j}{2}$ and $\bar l_{ij} = \frac{l_i + l_j}{2}$ are the mean values of the two nodes of edge $(i,j)$.
Using the above parameters, the elasticity matrix $K$ is assembled. Let $h=1/r$ denote the length of each network edge.

\begin{figure*}[ht!]
\centering
\subfigure[Regular square network.]{\includegraphics[trim={3.1cm 1.1cm 2.5cm 0.8cm},clip,width=0.31\linewidth]{NetworkExample2.eps}\label{NumExamA}}
\hspace{0.2cm}
\subfigure[Network with coarse scale grid representation.]{\includegraphics[trim={3.1cm 1.1cm 2.5cm 0.8cm},clip,width=0.31\linewidth]{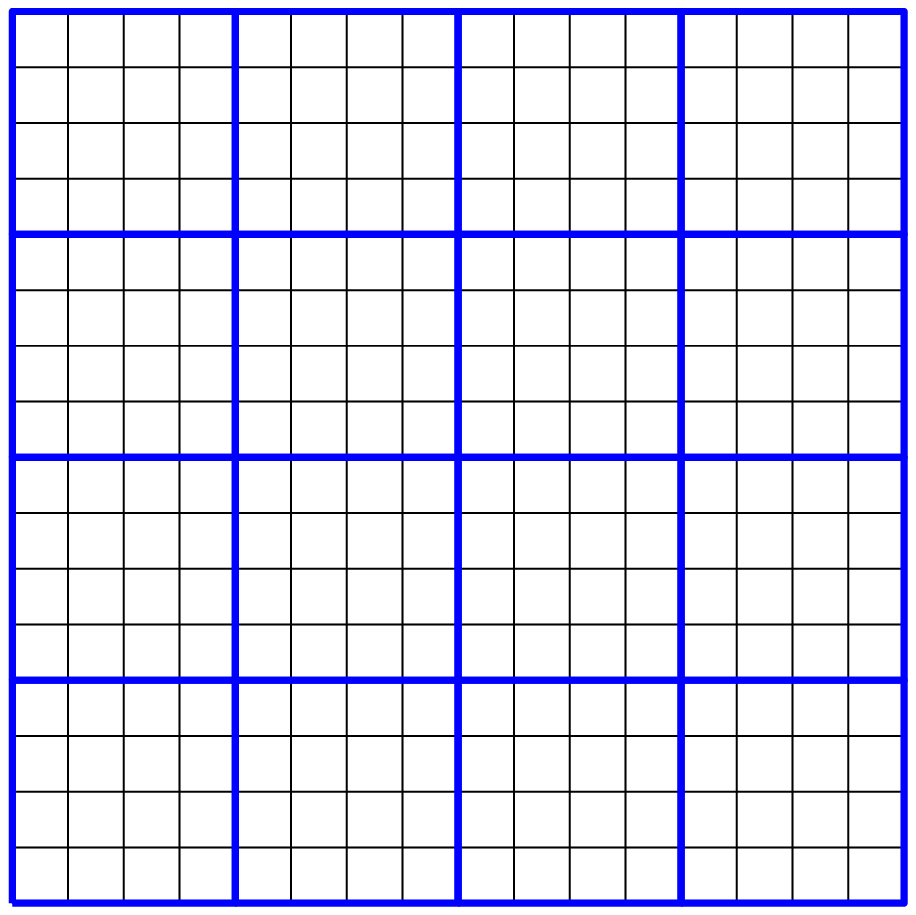}\label{NumExamB}}
\hspace{0.2cm}
\subfigure[Coarse scale basis function.]{\includegraphics[trim={4.0cm 1.5cm 3.5cm 2.7cm},clip,width=0.32\linewidth]{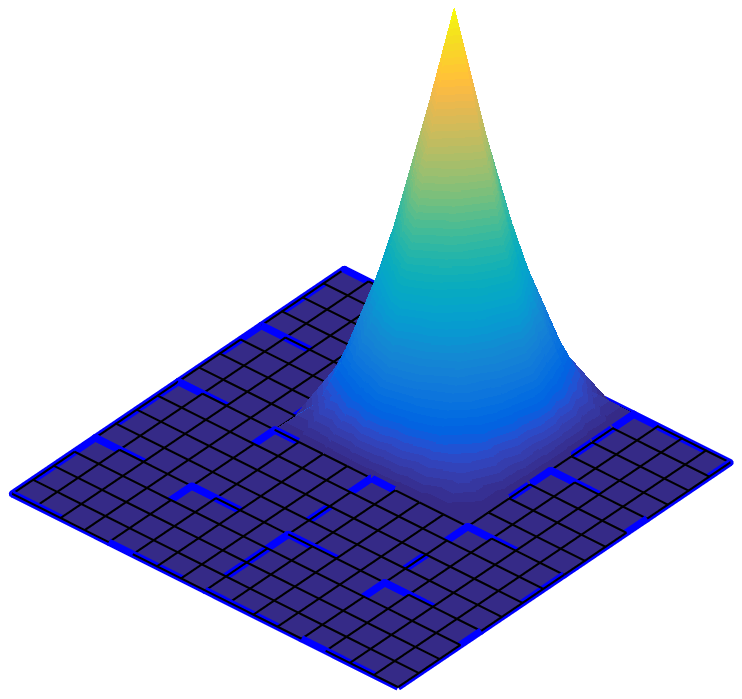}\label{NumExamC}}
\caption{{Example of a regular square network with $r=16$ and $R=4$. }}
\label{NumExam}
\end{figure*}

The first network problem, called the fixed boundary problem, is
\begin{align}
\label{FirstProblem}
\bar K_1 u = \bar F,
\end{align}
where $u(i)=0$ for all network nodes on the unit square boundary, and $F(i) =  \frac{1}{h^2}$. As mentioned earlier, $\bar K_1$ and $\bar F$ correspond to the modification of $K$ and $F$ by explicitly including the boundary conditions.

The second problem, called the displaced boundary problem, is
\begin{align}
\label{SecondProblem}
\bar K_2 u = 0
\end{align}
where $u(i)=0$ for all network nodes with $x=0$, and $u(i) = 0.1$ for all $x$-directed degrees of freedom with $x=1$.

To solve the two problems using the proposed numerical multiscale method, a coarse FEM grid is introduced. The grid is similar to the network but with $(R+1)^2$ nodes. The basis functions $\Lambda_i$ are chosen as classic bilinear. See Fig. \ref{NumExamB} for an illustration of the network and coarse FEM grid. Let $H= 1/R$ be the width of the coarse elements. With the described network geometry, the fixed boundary conditions correspond to fixation of coarse nodes at the boundary.

Each problem is solved with three different setups, first a basic setup with $l_i=\mu_i=1$. Secondly, by using a realization of random coefficients $l_i $ and $\mu_i$ sampled in $[0.1, 10]$ with uniform distribution. Thirdly, each node $i$ is displaced $[\delta x_i,\, \delta y_i]$ where $\delta x_i$ and $\delta y_i$ is randomly sampled in $[-0.4h, 0.4h]$ with uniform distribution. For nodes with initial coordinate $x=0$ or $x=1$, it is enforced that $\delta x_i = 0$, and similarly for nodes with $y=0$ or $y=1$, $\delta y_i=0$. In Fig. \ref{FineNetworkRandom}, a network with such random structure is shown for $r=32$.

\begin{figure*}[ht!]
\centering
\subfigure{\includegraphics[trim={3.1cm 1.1cm 2.6cm 0.8cm},clip,width=0.4\linewidth]{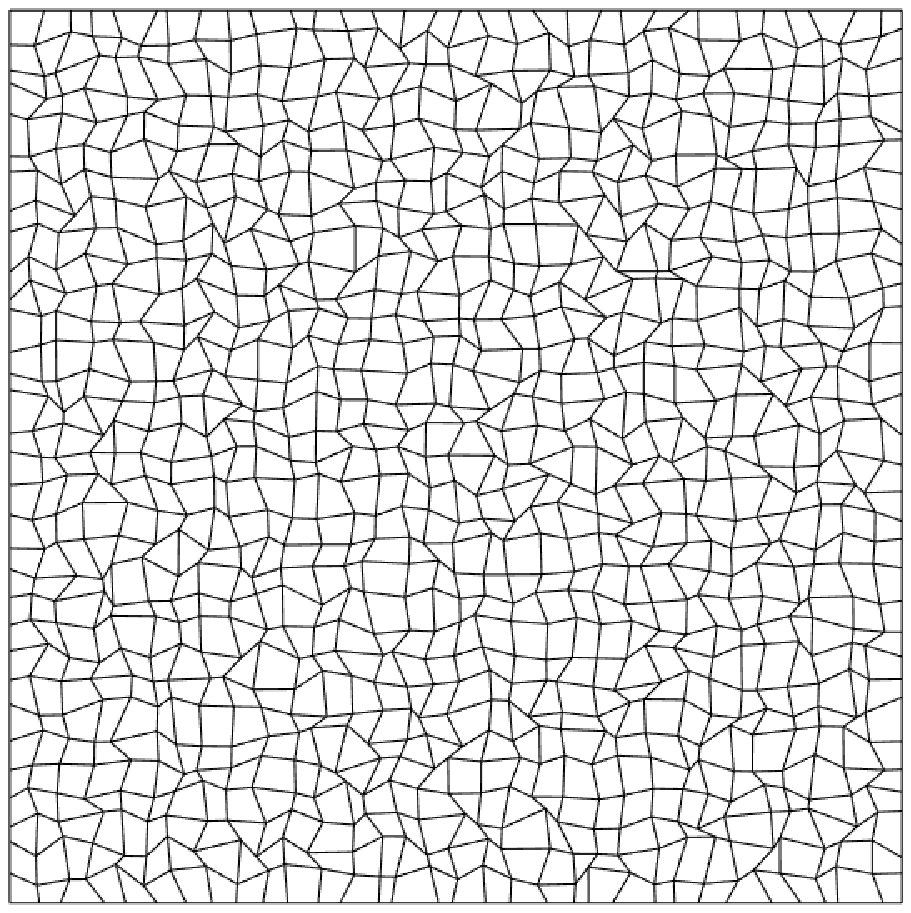}}
\caption{{Example of a unit square network with $r=32$ and randomly perturbed nodes. }}
\label{FineNetworkRandom}
\end{figure*}

The benefits of the multiscale method will be demonstrated by utilizing localization as described in Section \ref{Localization}, by introducing patches that the local modification problems are solved over. Which  degrees of freedom that are included in each patch set $\mathcal{N}_E$, are chosen based on the radius $\rho H$. A degree of freedom $i\in\mathcal{N}$ will be in $\mathcal{N}_E$, if $|p_i-c_E|\leq \rho H$, where $c_E\in\mathbb{R}^2$ is the center of  element $E$. Hence patches will have the circular form as was illustrated in Fig. \ref{PatchExam1a}. The rapid decay of the modified basis $\{\lambda_i-\phi_i\}_{i\in\mathcal{M}}$ is demonstrated by solving the modification $\tilde\phi_i$ with different  $\rho$ and computing the relative error $\vertiii{\phi_i-\tilde\phi_i}/\vertiii{\phi_i}$. The degree of freedom $i$ is chosen as one of the central nodes but the trend is similar for all nodes. The resulting errors for the first  problem \eqref{FirstProblem} are seen in Fig. \ref{PatchExam}.

\begin{figure*}[ht!]
\centering
\subfigure{\includegraphics[trim={0.0cm 0.1cm 1.5cm 0.6cm},clip,width=0.45\linewidth]{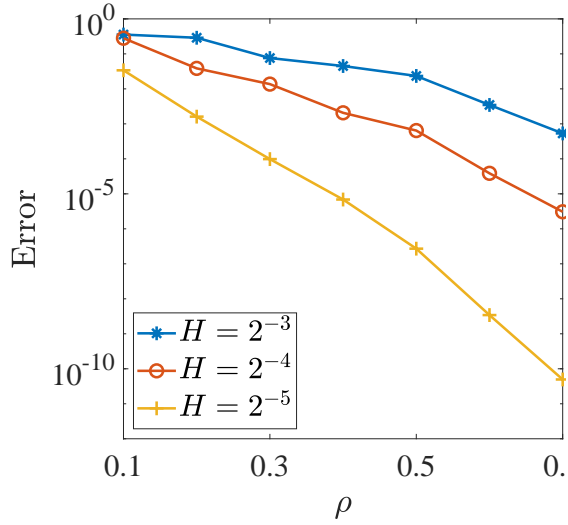}}
\caption{{Relative error $\vertiii{\phi_i-\tilde\phi_i}/\vertiii{\phi_i}$ for a localized modified basis function $\tilde \phi_i$ for a central node for the first problem \eqref{FirstProblem}.}}
\label{PatchExam}
\end{figure*}

Next the two problems \eqref{FirstProblem} and \eqref{SecondProblem} are solved using $r= 2^7 = 128$ for different coarse grids with $R = 2, 4, 8, 16, 32$. The problems are solved using localization with patch radius $\rho H =C H \log_2(H^{-1})=k/2^k$ where $k=1,2,3,4,5$. For the first problem \eqref{FirstProblem}, the constant is chosen to $C=1$, and for the second problem \eqref{SecondProblem} it is $C=1.5$. The second problem, with non-zero fixed displacement, is solved with correction as described in Section \ref{NonzeroFixation}. The resulting error curves for the two problem types and their three different setups are shown in Fig. \ref{Fixed} and Fig. \ref{Displaced}. In some setups the errors can be compared with the so-called FEM-error, corresponding to solving the multiscale problem \eqref{MatrixMSProblemLocal} with non-modified basis $B_H$ instead of the modified multiscale basis $\tilde B_\textnormal{ms}$. It can be seen that this FEM-solution behaves poorly for the setup with random coefficients.

\begin{figure*}[ht!]
\centering
\subfigure[$K$-error basic.]{\includegraphics[trim={0cm 0cm 0cm 0cm},clip,width=0.32\linewidth]{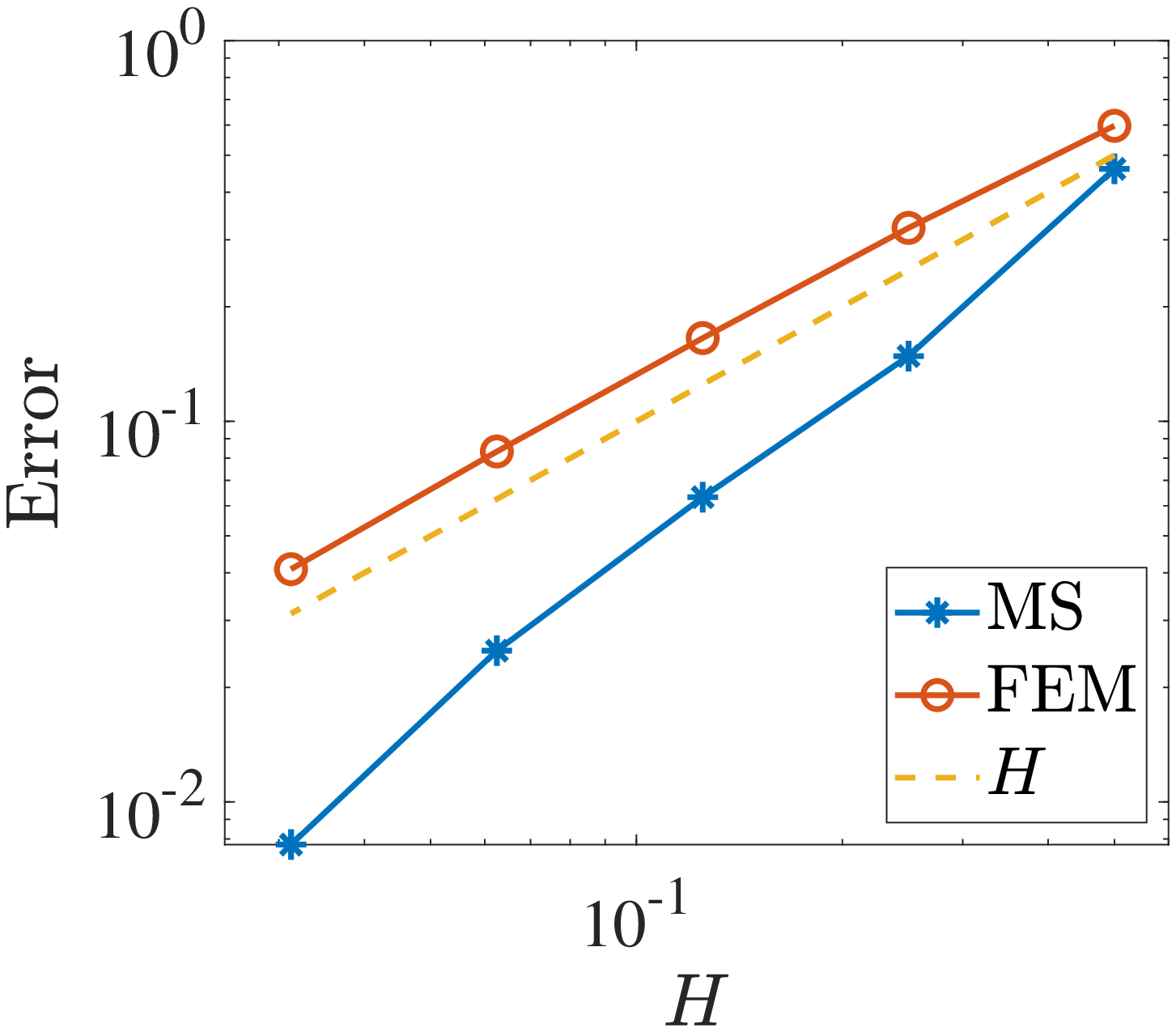}}
\subfigure[$K$-error random coefficients.]{\includegraphics[trim={0cm 0cm 0cm 0cm},clip,width=0.32\linewidth]{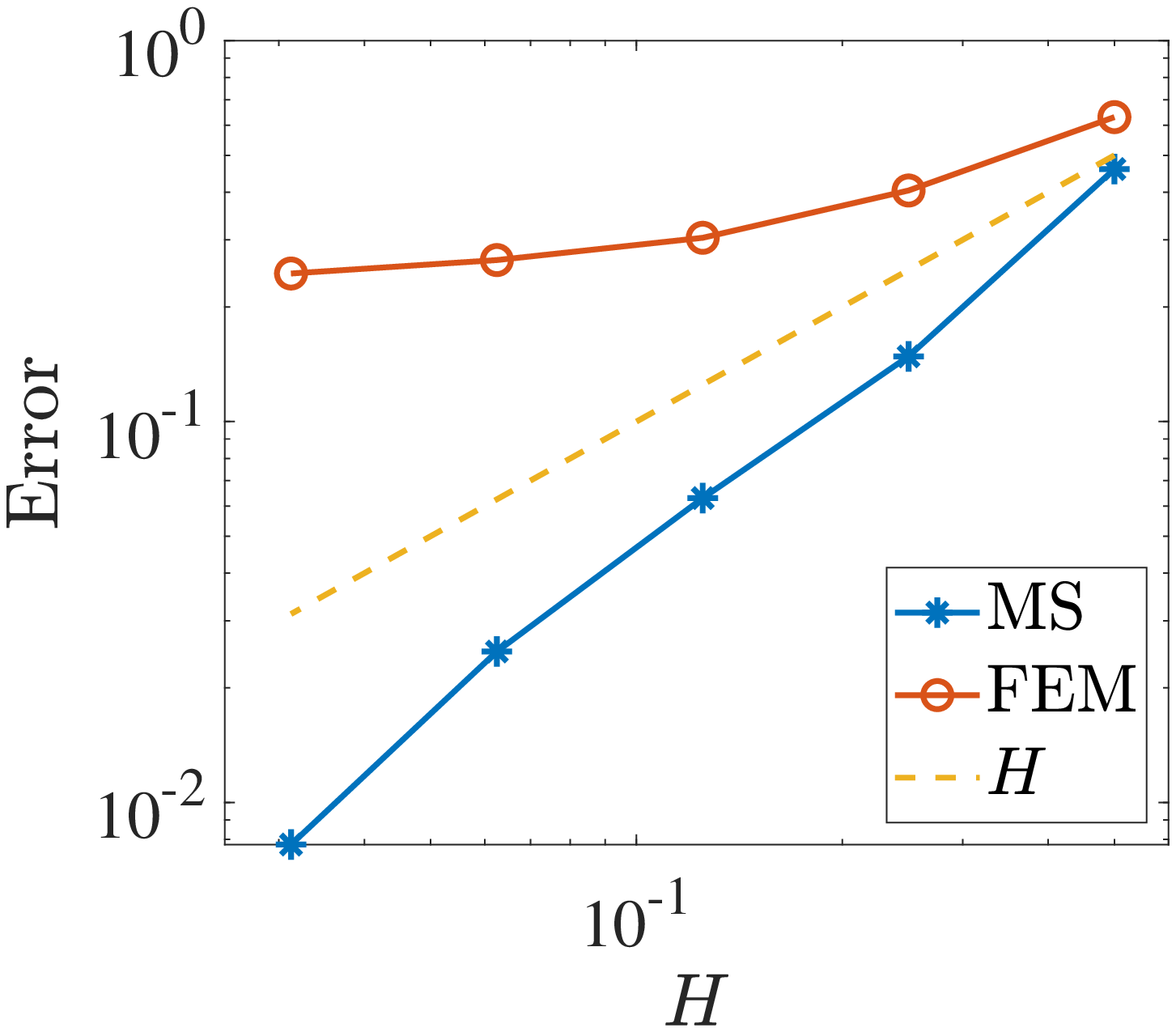}}
\subfigure[$K$-error random structure.]{\includegraphics[trim={0cm 0cm 0cm 0cm},clip,width=0.32\linewidth]{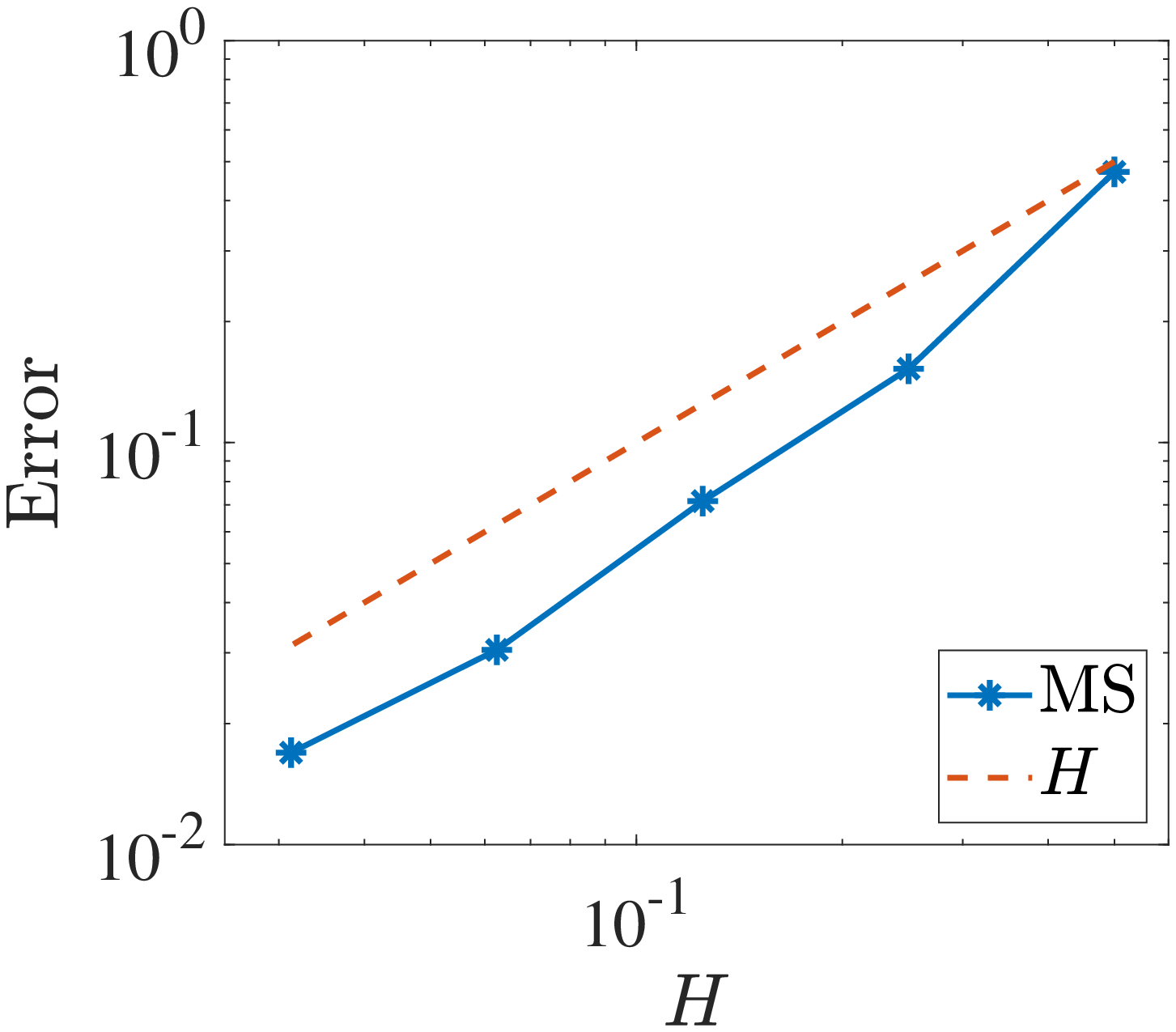}}
\subfigure[$L^2$-error basic.]{\includegraphics[trim={0cm 0cm 0cm 0cm},clip,width=0.32\linewidth]{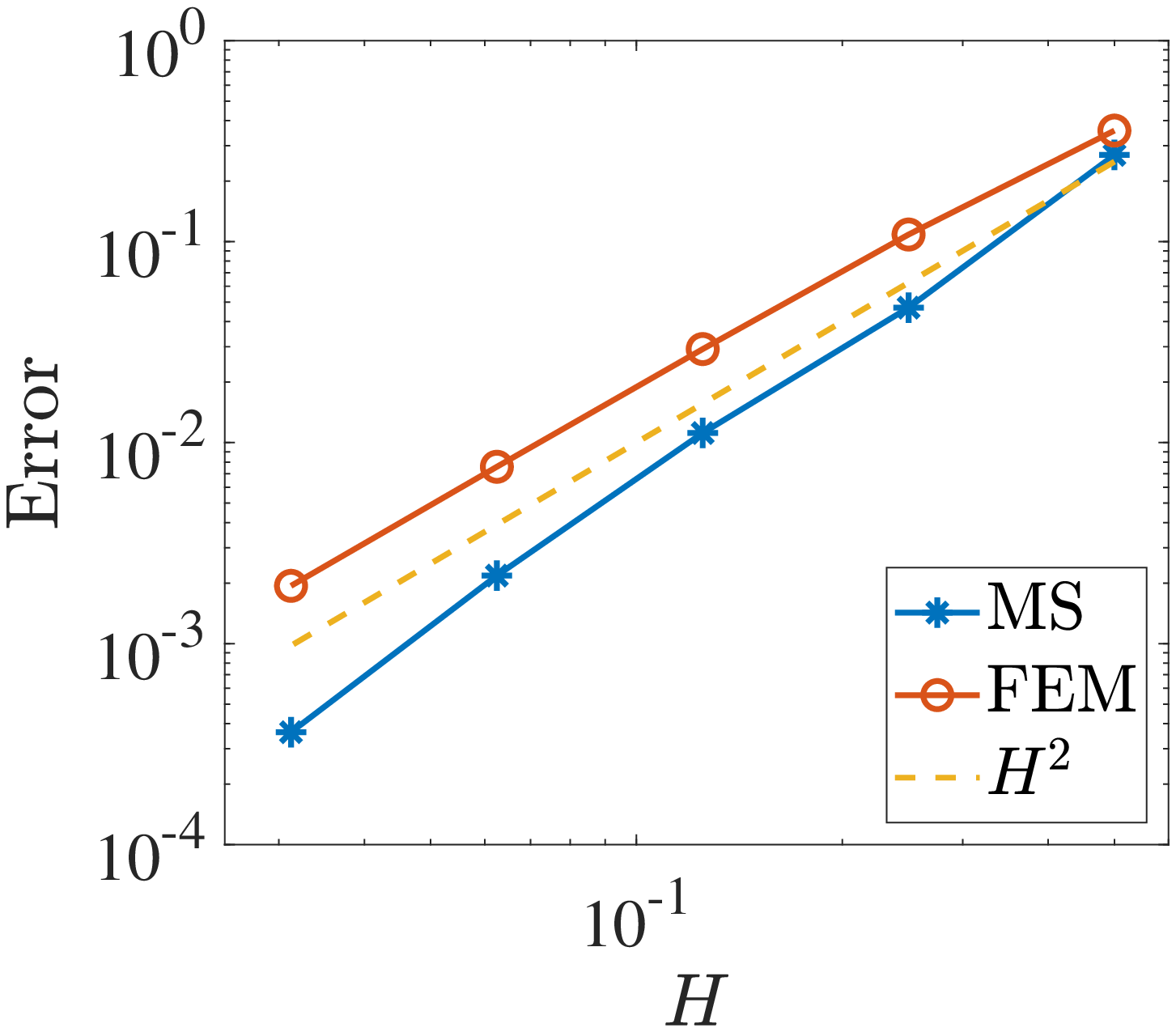}}
\subfigure[$L^2$-error random coefficients.]{\includegraphics[trim={0cm 0cm 0cm 0cm},clip,width=0.32\linewidth]{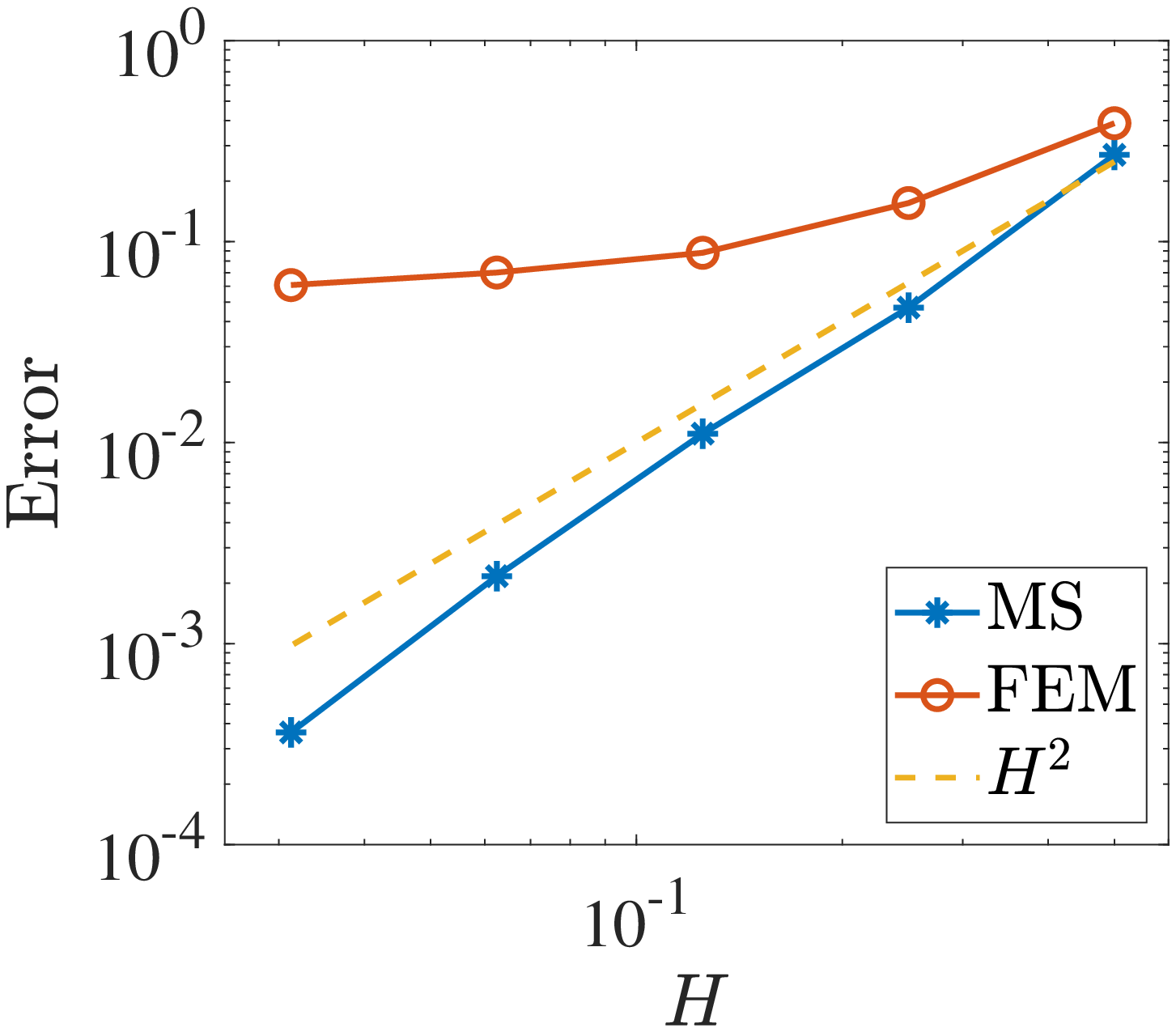}}
\subfigure[$L^2$-error random structure.]{\includegraphics[trim={0cm 0cm 0cm 0cm},clip,width=0.32\linewidth]{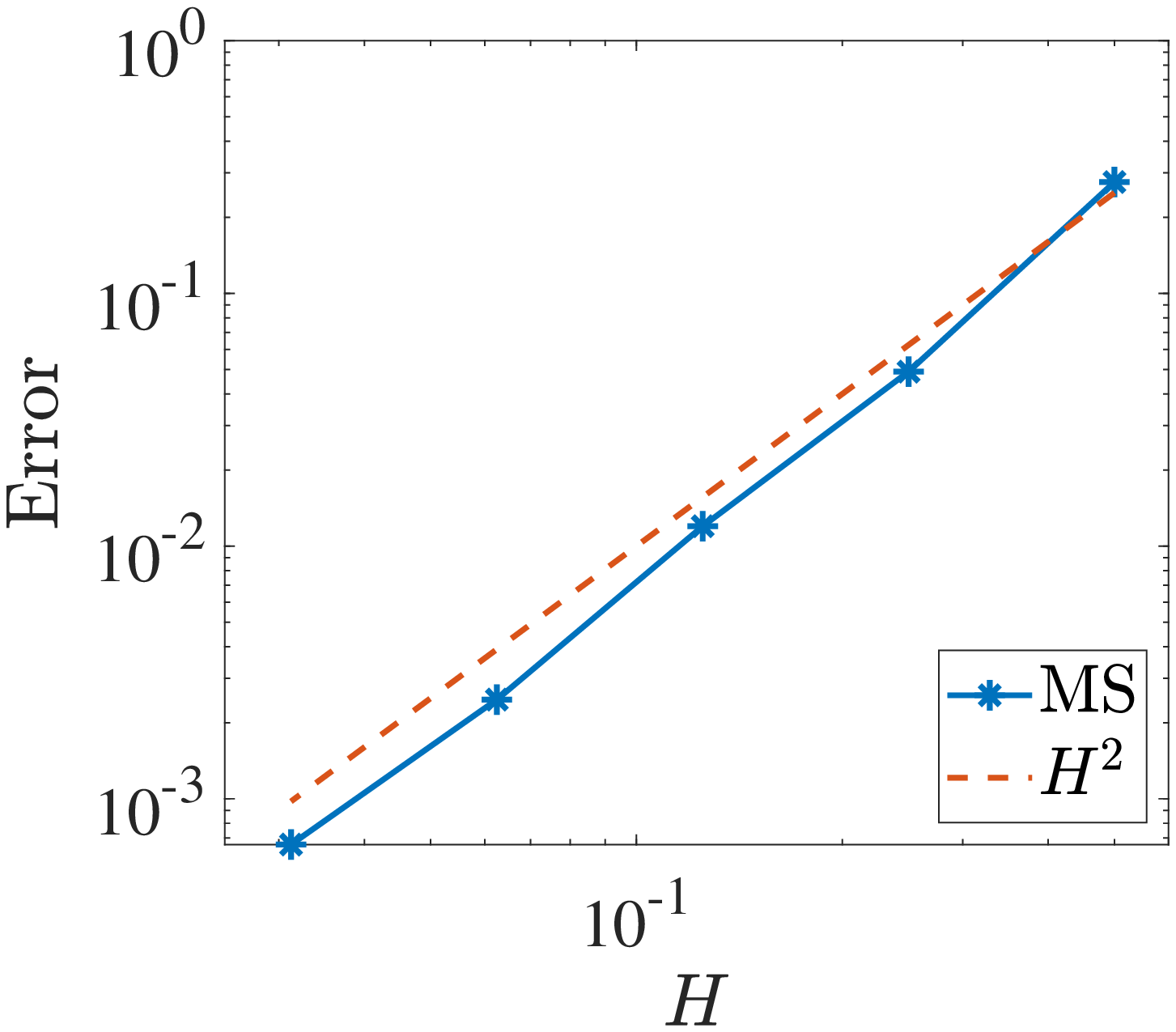}}
\caption{{Errors for the first problem \eqref{FirstProblem}.  The coarse mesh size $H$ and its square $H^2$ are included in the plots to clarify convergence rates. For the basic setup and the setup with random coefficients, also the so-called FEM-error is included. }}
\label{Fixed}
\end{figure*}

\begin{figure*}[ht!]
\centering
\subfigure[$K$-error basic.]{\includegraphics[trim={0cm 0cm 0cm 0cm},clip,width=0.32\linewidth]{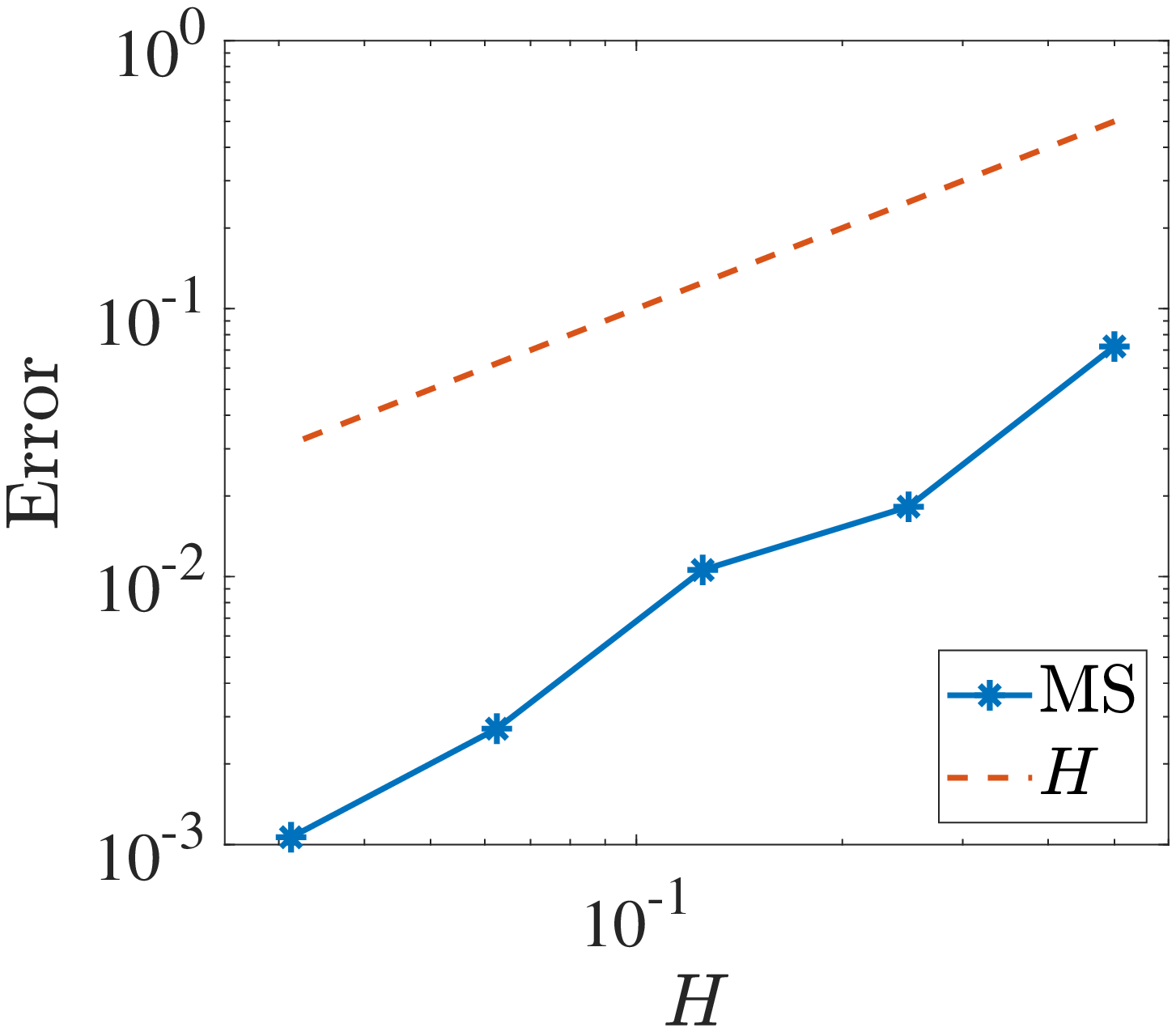}}
\subfigure[$K$-error random coefficients.]{\includegraphics[trim={0cm 0cm 0cm 0cm},clip,width=0.32\linewidth]{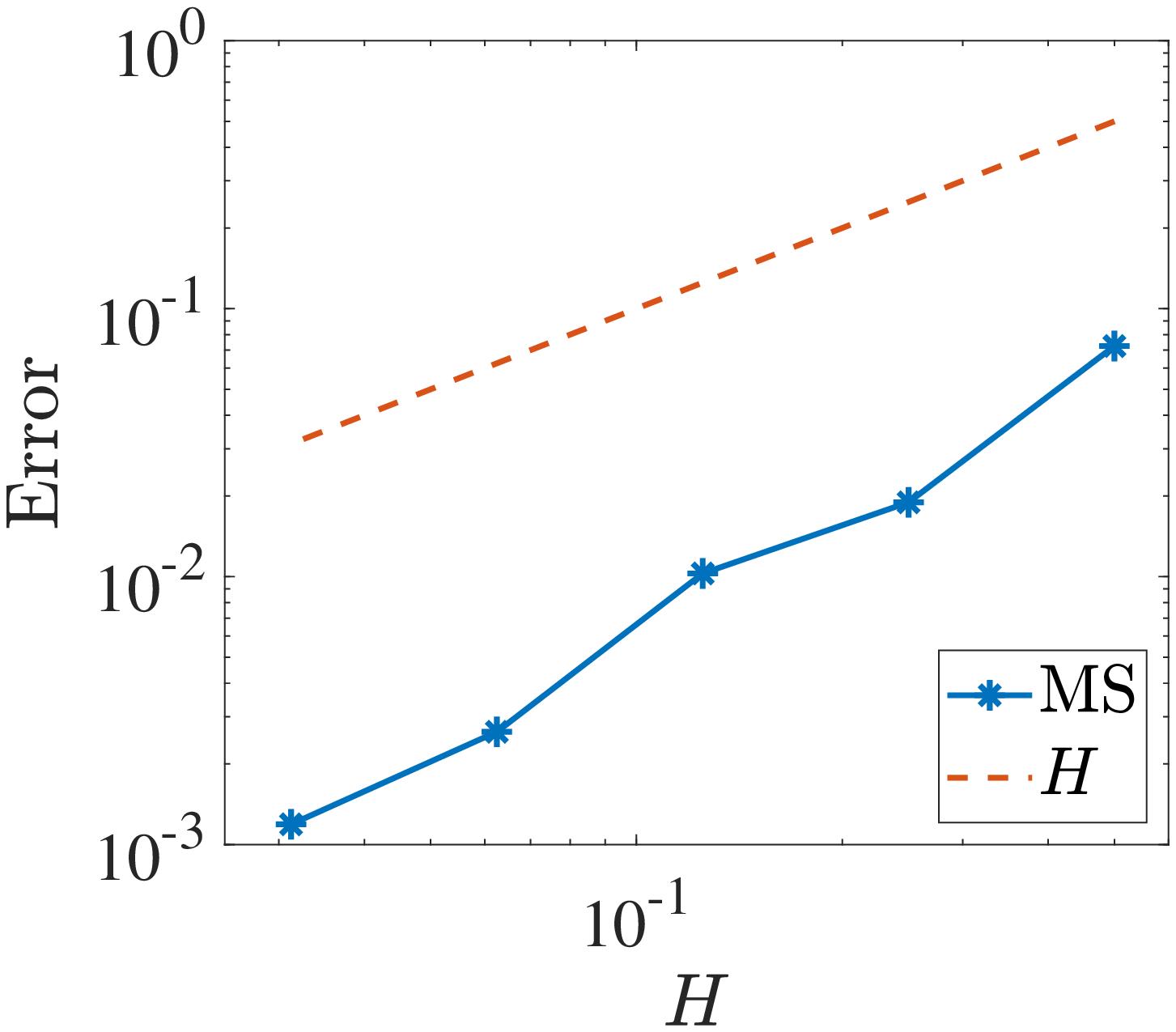}}
\subfigure[$K$-error random structure.]{\includegraphics[trim={0cm 0cm 0cm 0cm},clip,width=0.32\linewidth]{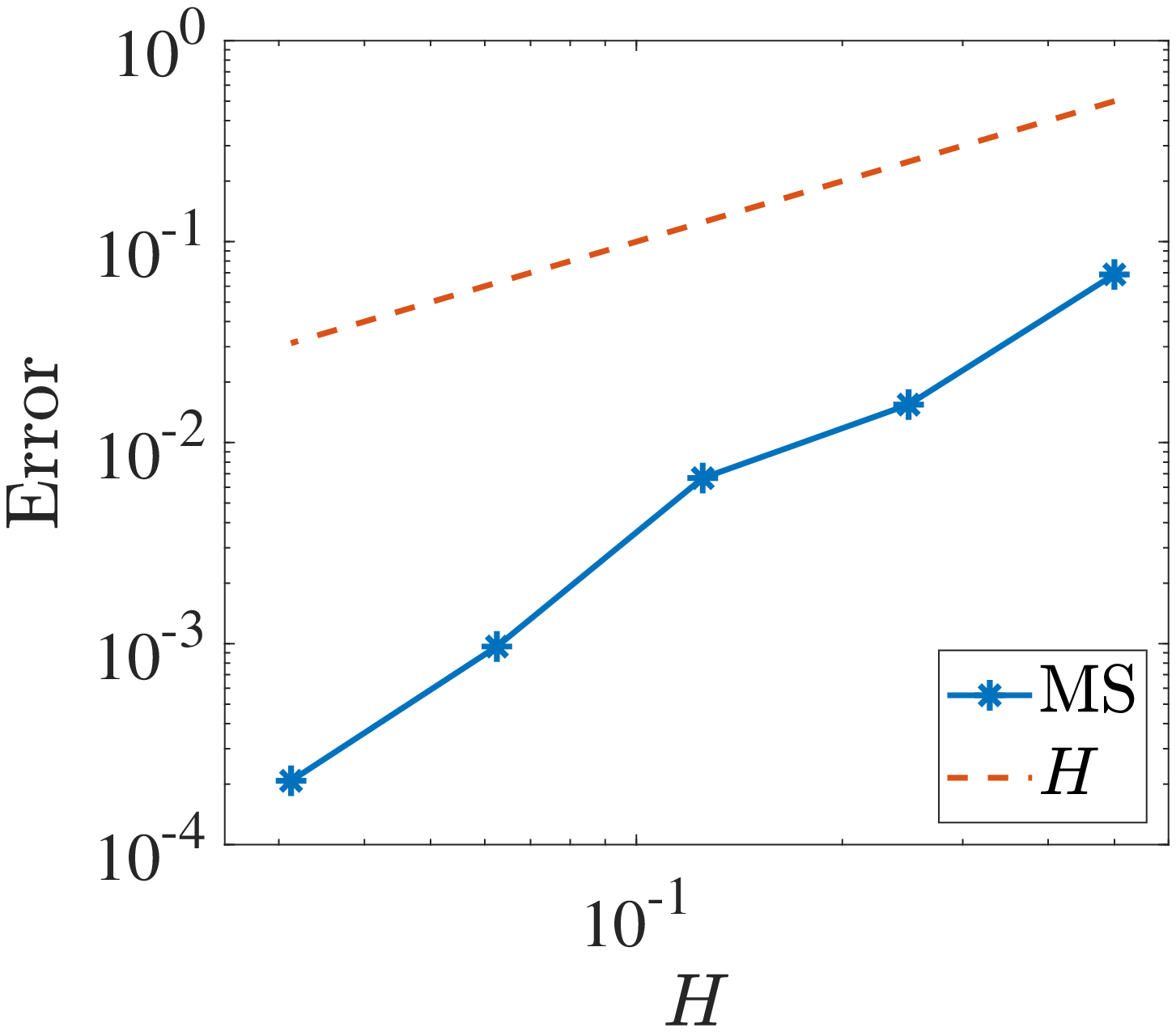}}
\subfigure[$L^2$-error basic.]{\includegraphics[trim={0cm 0cm 0cm 0cm},clip,width=0.3\linewidth]{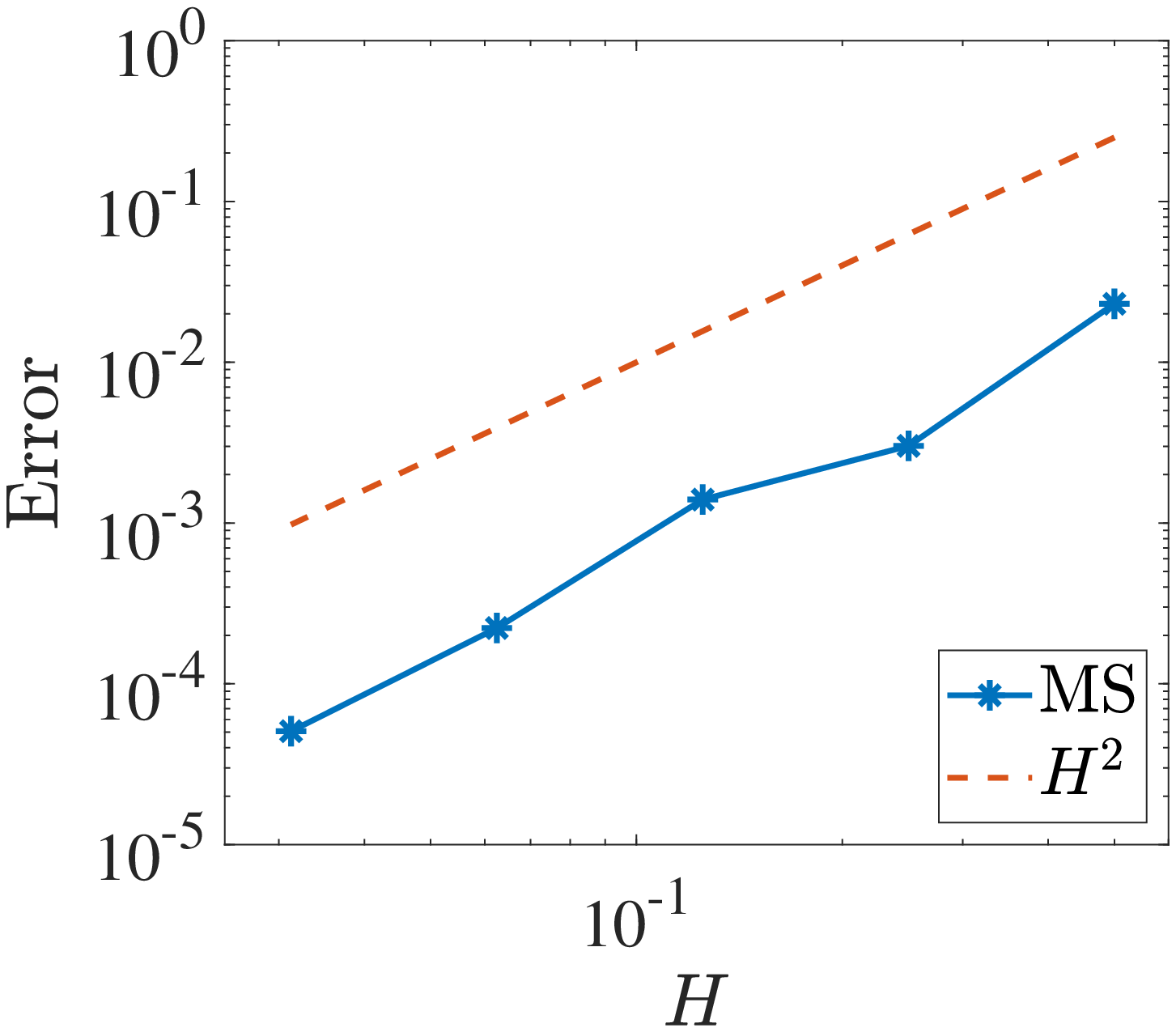}}
\subfigure[$L^2$-error random coefficients.]{\includegraphics[trim={0cm 0cm 0cm 0cm},clip,width=0.32\linewidth]{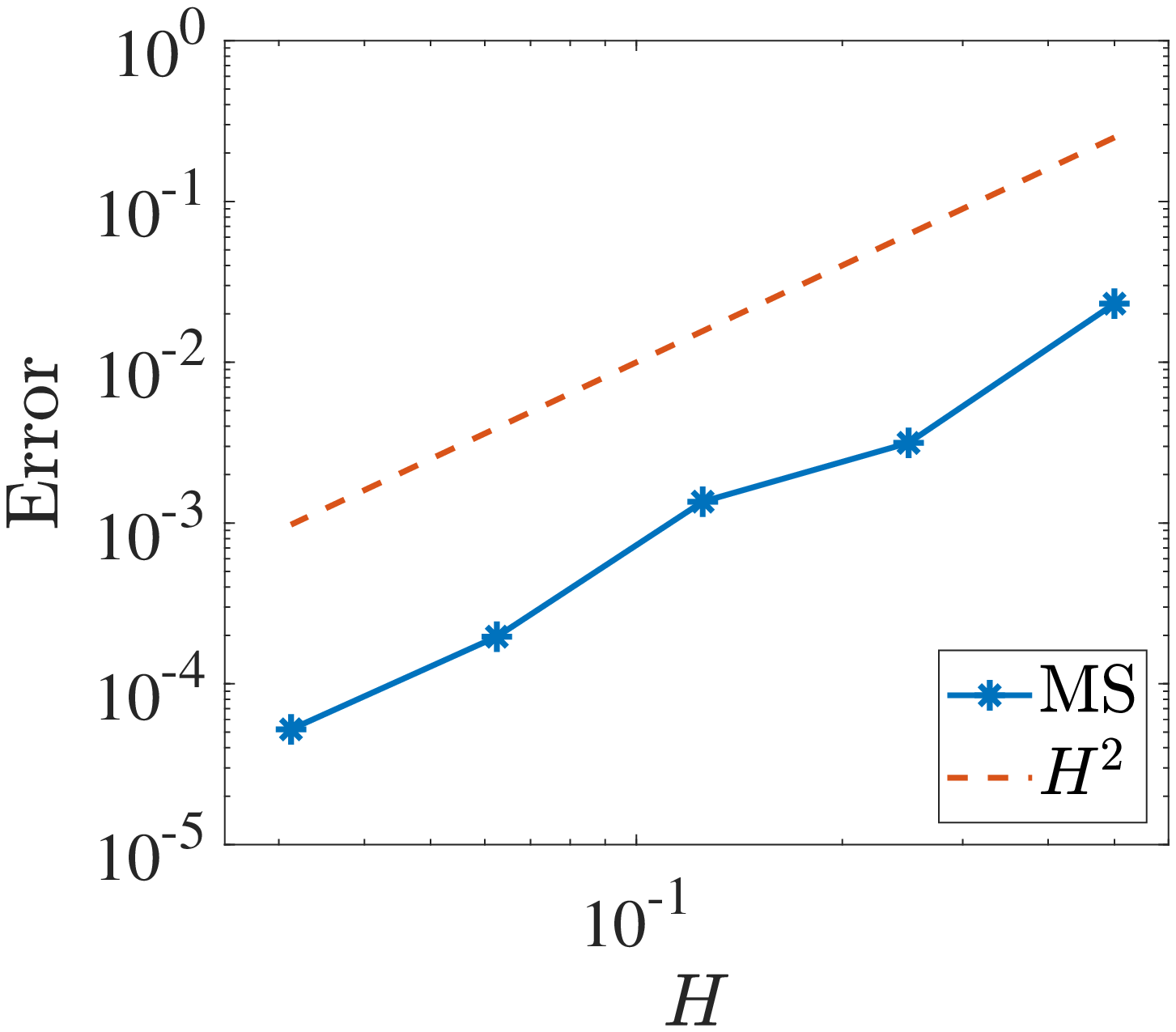}}
\subfigure[$L^2$-error random structure.]{\includegraphics[trim={0cm 0cm 0cm 0cm},clip,width=0.32\linewidth]{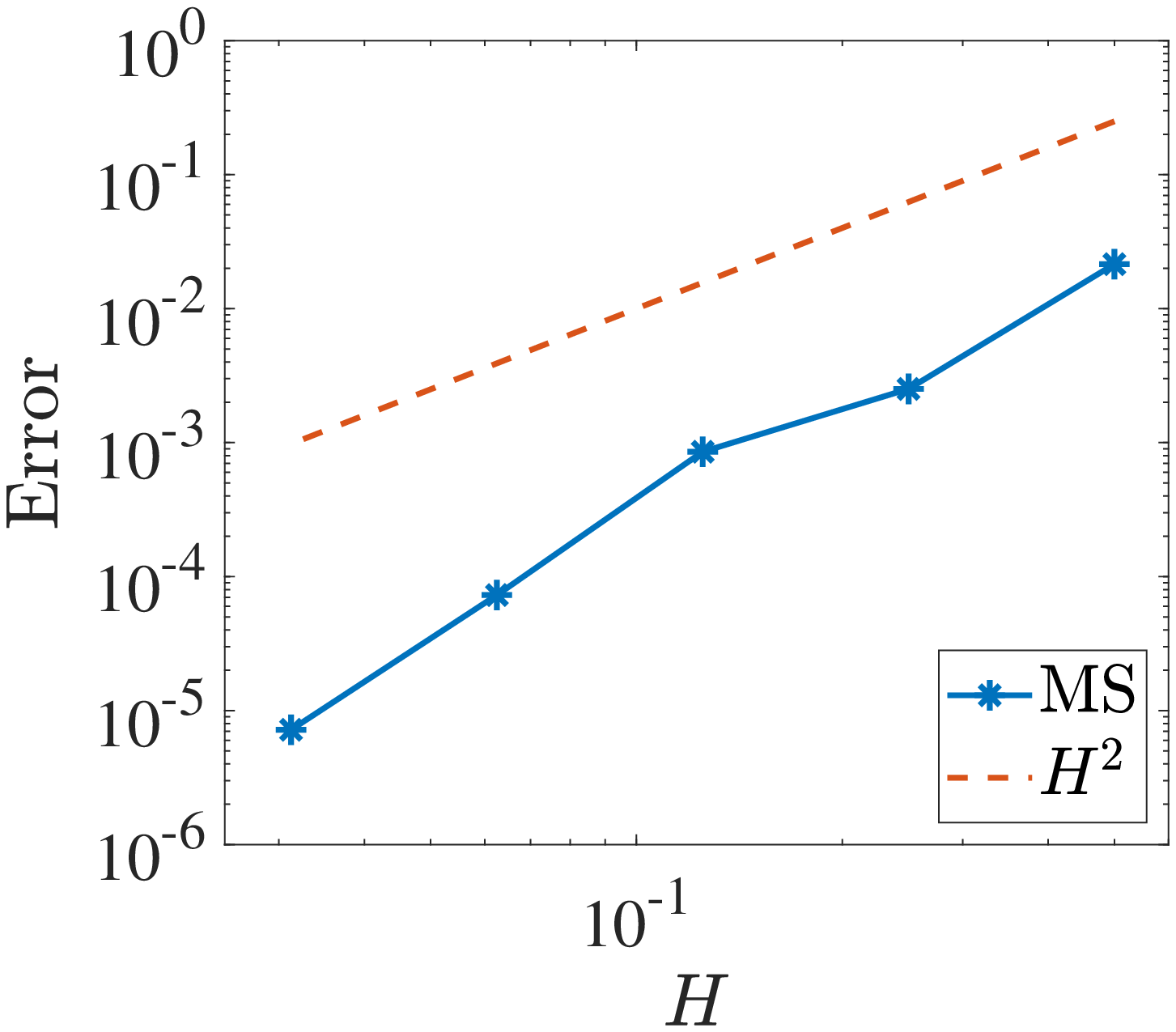}}
\caption{{Errors for the second problem \eqref{SecondProblem} with correction as in Section \ref{Displaced}. The coarse mesh size $H$ and its square $H^2$ are included in the plots to clarify convergence rates.}}
\label{Displaced}
\end{figure*}

\section{Conclusion and future work }
\label{Conclusions}
In this paper a numerical multiscale method for discrete networks is proposed. For a set of different numerical examples, the convergence rates of the proposed method are examined. For regular networks with low connectivity variation the method resembles the convergence rates of the ordinary FEM. For networks with randomly varying connectivity it is shown that the multiscale method performs better than ordinary FEM. The method is moreover used to solve network problems with  random structure, indicating error convergence rates at least linear in  energy norm and quadratic in  $L^2$-norm.

 A challenging theoretical problem for the future is to extend the result in \eqref{eq:decay} beyond finite element based discretization to more general networks, including the presented network model.  Further on, other fiber network models, for example beam based, as well as three-dimensional, should be considered. Thereby the presented numerical upscaling method will be applied to  realistic macroscale paper networks, investigating the problem size which can be studied and the computational efficiency of the proposed method. Moreover, the proposed method will be used together with the paper forming simulation framework presented in \cite{Mark1, Mark2, Lic} to study virtual paper sheets and macroscale mechanical properties such as tensile strength, tensile stiffness, bending resistance, $z$-strength and fracture propagation.

\subsubsection*{Acknowledgements}
This work is a part of the ISOP (Innovative Simulation of Paper) project which is performed by a consortium consisting of Albany International, Stora Enso and Fraunhofer-Chalmers Centre.


\begin{thebibliography}{}

\bibitem{Chu}
Chu, J. Engquist, B., Prodanovi\'c, M., Tsai, R.: A multiscale method coupling network and continuum models in porous media I: steady-state single phase flow.  Multiscale Model. Simul. \textbf{10}, 515-549 (2012)

\bibitem{DellaRossa}
Della Rossa, F., D'Angelo, C., Quarteroni, F.: A distributed model of traffic flows on extended regions. Netw. Heterog. Media \textbf{5}, 525-544 (2010)

\bibitem{Efendiev}
Efendiev, Y., Galvis, J., Hou, T.Y.: Generalized multiscale finite element methods (GMsFEM), Journal of Computational Physics archive,  \textbf{251}, 116-135 (2013)

\bibitem{LOD2}
Engwer, C., Henning, P, M\r{a}lqvist, A., Peterseim, D.: Efficient implementation of the Localized Orthogonal Decomposition method. arXiv:1602.01658v2 (2017)

\bibitem{Ewing}
Ewing, R., Iliev, O., Lazarov, R., Rybak, I., Willems, J.: A simplified method for upscaling composite materials with high contrast of the conductivity. SIAM J. Sci. Comput. \textbf{31}, 2568-2586 (2009)

\bibitem{HM2014}
Henning, P., M\r{a}lqvist, A.: Localized orthogonal decomposition techniques for boundary value problems, SIAM J. Sci. Comp. \textbf{36},  A1609-A1634 (2014)

\bibitem{HePe2016}
Henning, P., Persson, A.: A Multiscale Method for linear elasticity reducing Poisson locking. Comput. Methods Appl. Mech. Eng. \textbf{310}, 156-171 (2016)

\bibitem{MsFEM}
Hou, T.Y., Wu, X.: A multiscale finite element method for elliptic problems in composite materials and porous media.  J. Comput. Phys. \textbf{134}, 169-189 (1997)

\bibitem{Hughes}
Hughes, T.J.R., Feij\'oo, G.R., Mazzei, L., Quincy, J.B.: The variational multiscale method - a paradigm for computational mechanics. Comput. Methods Appl. Mech. Engrg. \textbf{166}, 3-24 (1998)

\bibitem{Hagglund}
H$\ddot{\textnormal{a}}$gglund, R., Isaksson, P.: On the coupling between macroscopic material degradation and interfiber bond fracture in an idealized fiber network. Int. J. Solids Struct. \textbf{45}, 868-878 (2007)

\bibitem{Iliev}
Iliev, O., Lazarov, R., Willems, J.: Fast numerical upscaling of heat equation for fibrous materials. Comput. Visual. Sci. \textbf{13}, 275-285 (2010)

\bibitem{Lic}
Kettil, G. {A Novel Fiber Interaction Method for Simulation of Early Paper Forming}, Licentiate thesis, Chalmers University of Technology, (2016)

\bibitem{Kulachenko}
Kulachenko, A., Uesaka, T.: Direct simulations of fiber network deformation and failure. Mech. Mater. \textbf{51}, 1-14 (2012)

\bibitem{Mark1} Mark, A., Svenning, E., R. Rundqvist, F. Edelvik, Glatt, E., Rief, S., Wiegmann, A., Fredlund, M., Lai, R., Martinsson, L., Nyman, U.: Microstructure Simulation of Early Paper Forming Using Immersed Boundary Methods, TAPPI J. \textbf{10}, 23-30 (2011)

\bibitem{LOD1}
M\r{a}lqvist, A., Peterseim, D.,: Localization of elliptic multiscale problems. Math. Comp. \textbf{83}, 2583-2603 (2014)

\bibitem{Ostoja}
Ostoja, M., Shenk, P.Y., Alzebdeh, K.: Spring network models in elasticity and fracture of composite and polycrystal. Comput. Mater. Sci. \textbf{7}, 82-93 (1996)

\bibitem{Owhadi}
Owhadi, H., Zhang,  L., Berlyand, L.: Polyharmonic homogenization, rough polyharmonic splines and sparse super-localization. ESAIM Math. Model. Numer. Anal. \textbf{48}, 517-552 (2013)

\bibitem{Raisanen}
R$\ddot{\textnormal{a}}$is$\ddot{\textnormal{a}}$nen, V. I., Alava, M.J., Nieminen, R.M., Niskanen, K.J.: Elastic-plastic behaviour in fibre networks. Nord. Pulp Pap. Res. J. \textbf{11}, 243-248 (1996)

\bibitem{Mark2} Svenning, E., Mark, A., Edelvik, F., Glatt, E., Rief, S., Wiegmann, A., Martinsson, L., Lai, R., Fredlund, M., Nyman, U.: Multiphase Simulation of Fiber Suspension Flows Using Immersed Boundary Methods. Nord. Pulp Pap. Res. J. \textbf{27}, 184-191 (2012)

\bibitem{HMM}
Weinan, E., Engquist, B.: The heterogeneous multiscale methods. Commun. Math. Sci. \textbf{1}, 87-132 (2003)

\bibitem{Beex}
Wilbrink, D.V, Beex, L.A.A, Peerlings, R.H.J.: A discrete network model for bond failure and frictional sliding in fibrous materials. Int. J. Solids Struct. \textbf{50}, 1354-1363 (2013)


\end{thebibliography}


\end{document}